\documentclass[11pt,letterpaper]{amsart}
\usepackage{amssymb, enumitem}
\usepackage[all]{xy}
\usepackage{tikz} 
\usepackage{hyperref}
\usepackage{aliascnt}
\usepackage{verbatim}
\usepackage[english]{babel}
\usepackage{enumitem}
\usepackage{mathtools}

\def\today{\number\day\space\ifcase\month\or   January\or February\or
   March\or April\or May\or June\or   July\or August\or September\or
   October\or November\or December\fi\   \number\year}

\theoremstyle{definition}
\newtheorem{lma}{Lemma}[section]

\newaliascnt{thmCt}{lma}
\newtheorem{thm}[thmCt]{Theorem}
\aliascntresetthe{thmCt}

\newaliascnt{corCt}{lma}
\newtheorem{cor}[corCt]{Corollary}
\aliascntresetthe{corCt}

\newaliascnt{cnjCt}{lma}

\aliascntresetthe{cnjCt}

\newaliascnt{propCt}{lma}
\newtheorem{prop}[propCt]{Proposition}
\aliascntresetthe{propCt}

\newtheorem*{thm*}{Theorem}
\newtheorem*{cor*}{Corollary}
\newtheorem*{prop*}{Proposition}

\newcounter{theoremintro}
\newtheorem{defintro}[theoremintro]{Definition}
\newtheorem{thmintro}[theoremintro]{Theorem}

\newtheorem{cnjintro}[theoremintro]{Conjecture}

\newaliascnt{pgrCt}{lma}

\aliascntresetthe{pgrCt}

\newaliascnt{dfCt}{lma}
\newtheorem{df}[dfCt]{Definition}
\aliascntresetthe{dfCt}

\newaliascnt{remCt}{lma}
\newtheorem{rem}[remCt]{Remark}
\aliascntresetthe{remCt}

\newaliascnt{remsCt}{lma}

\aliascntresetthe{remsCt}

\newaliascnt{egCt}{lma}
\newtheorem{eg}[egCt]{Example}
\aliascntresetthe{egCt}

\newaliascnt{egsCt}{lma}

\aliascntresetthe{egsCt}

\newaliascnt{qstCt}{lma}

\aliascntresetthe{qstCt}

\newaliascnt{pbmCt}{lma}

\aliascntresetthe{pbmCt}

\newaliascnt{notaCt}{lma}

\aliascntresetthe{notaCt}

\newcommand{\beq}{\begin{equation}}
\newcommand{\eeq}{\end{equation}}
\newcommand{\beqa}{\begin{eqnarray*}}
\newcommand{\eeqa}{\end{eqnarray*}}
\newcommand{\bal}{\begin{align*}}
\newcommand{\eal}{\end{align*}}
\newcommand{\bi}{\begin{itemize}}
\newcommand{\ei}{\end{itemize}}
\newcommand{\be}{\begin{enumerate}}
\newcommand{\ee}{\end{enumerate}}

\newcommand{\acton}{\curvearrowright}

\newcommand{\ep}{\varepsilon}

\newcommand{\del}{\partial}

\newcommand{\Z}{{\mathbb{Z}}}
\newcommand{\C}{{\mathbb{C}}}
\newcommand{\N}{{\mathbb{N}}}

\newcommand{\Prob}{{\mathrm{Prob}}}

\title{Classifiability of crossed products by nonamenable groups}

\thanks{The first named author was partially supported
by the Deutsche Forschungsgemeinschaft (DFG) through an eigene 
Stelle and by a Postdoctoral Research Fellowship
from the Humboldt Foundation. 
The second named author was supported by the project G085020N 
funded by the Research Foundation Flanders (FWO), by the generosity of Eric and Wendy Schmidt by recommendation of the Schmidt Futures program, by a Kreitman 
Foundation Fellowship, by an Israel Science Foundation grant 
no.~476/16, and by ERC Advanced Grant 834267 - AMAREC.  
All authors were partially supported by the DFG under Germany's 
Excellence Strategy EXC 2044 390685587, Mathematics M\"unster: 
Dynamics–Geometry–Structure, and through SFB 1442.
}

\author[Eusebio Gardella]{Eusebio Gardella}
\address{Eusebio Gardella
Department of Mathematical Sciences, Chalmers University of
Technology and University of Gothenburg, Gothenburg SE-412 96, Sweden.}
\email{gardella@chalmers.se}
\urladdr{www.math.chalmers.se/~gardella}

\author[Shirly Geffen]{Shirly Geffen}
\address{Shirly Geffen
Mathematisches Institut, Fachbereich Mathematik und Informatik der
Universit\"at M\"unster, Einsteinstrasse 62, 48149 M\"unster, Germany.}
\email{sgeffen@uni-muenster.de}
\urladdr{https://shirlygeffen.com/}

\author{Julian Kranz}
\address{Julian Kranz
Mathematisches Institut, Fachbereich Mathematik und Informatik der
Universit\"at M\"unster, Einsteinstrasse 62, 48149 M\"unster, Germany.}
\email{julian.kranz@uni-muenster.de}
\urladdr{https://www.uni-muenster.de/IVV5WS/WebHop/user/j\_kran05/}

\author{Petr Naryshkin}
\address{Petr Naryshkin
Mathematisches Institut, Fachbereich Mathematik und Informatik der
Universit\"at M\"unster, Einsteinstrasse 62, 48149 M\"unster, Germany.}
\email{pnaryshk@uni-muenster.de}
\urladdr{http://petrnaryshkin.wordpress.com}

\begin{document}
 
\begin{abstract}
We show that all amenable, minimal actions of a large class of 
nonamenable countable groups on compact metric spaces have dynamical comparison. This 
class includes all nonamenable hyperbolic groups, 
many HNN-extensions, nonamenable
Baumslag-Solitar groups, a large class of amalgamated free 
products, lattices in many Lie groups, $\widetilde{A}_2$-groups, 
as well as direct products of the 
above with arbitrary countable groups. As a consequence, crossed products 
by amenable, minimal and topologically free actions of such groups on 
compact metric spaces are Kirchberg algebras in the UCT class, 
and are therefore classified by $K$-theory. 
\end{abstract}

\maketitle
\vspace{-.6cm}
\tableofcontents
\vspace{-.9cm}


\renewcommand*{\thetheoremintro}{\Alph{theoremintro}}
\section{Introduction}

One of the most remarkable achievements in $C^*$-algebra 
theory in the last decade was the completion of the 
classification programme initiated by George Elliott over 30 years
ago. The outcome is the combination of the work of 
a large number of mathematicians over several decades, and can be
phrased as follows:

\begin{thm*} [Classification]
Simple, separable, unital, nuclear, $\mathcal Z$-stable 
$C^*$-algebras satisfying the Universal Coefficient Theorem (UCT) 
are classified by 
the Elliott invariant ($K$-theory and traces).
\end{thm*}

By Kirchberg's dichotomy, \cite[Theorem~4.1.10]{Ror_classification_2002},
a $C^*$-algebra satisfying the assumptions of the above theorem
(also called classifiable)
is either stably finite or purely infinite.
The currently available proof of the classification 
theorem considers the stably finite and purely infinite cases
separately: while the purely infinite case was settled over 
20 years ago by Kirchberg and Phillips (see
\cite{Phi_classification_2000}), the stably finite case was 
only settled in the last 5 years as a combination of 
\cite{EGLN_classification_2016, TikWhiWin_quasidiagonality_2017,
CETWW_nuclear_2021}.
We refer the reader to Winter's ICM address 
\cite{Win_structure_2018} for a recent survey and further
references on the topic. 

With such a powerful classification theorem at our disposal, 
it becomes an imperative task to identify interesting classes
of $C^*$-algebras to which it can be applied. One
of the most natural families of $C^*$-algebras arises from
topological dynamics via the crossed product construction.
In recent years, a lot of work has been done to 
establish dynamical criteria for an action $G\acton X$ 
of a countable group on a compact metric space 
that ensure that the associated
crossed product $C(X)\rtimes G$ satisfies the assumptions 
of the classification theorem. 
Unitality and separability of $C(X)\rtimes G$ are automatic,
while
nuclearity of $C(X)\rtimes G$ is equivalent to amenability
of $G\acton X$ (see \autoref{df:AmenAction}). Moreover, 
if $G\acton X$ is amenable, then $C(X)\rtimes G$ automatically
satisfies the UCT by \cite[Theorem~10.9]{Tu_conjecture_1999}, 
and it is simple if and only if
$G\acton X$ is minimal and topologically free by
\cite[Theorem~2]{ArcSpi_topologically_1994}. In particular, amenability, 
minimality, and topological freeness are \emph{necessary} conditions
for classifiability of $C(X)\rtimes G$, and
it remains to decide when $C(X)\rtimes G$ is
$\mathcal{Z}$-stable. 
Kirchberg's dichotomy takes a particularly nice form in this setting,
since a nuclear crossed product $C(X)\rtimes G$ is stably finite if
and only if $G$ is amenable. (This is a combination of celebrated 
results in \cite{Cun_dimension_1978, BlaHan_dimension_1982, 
Haa_quasitraces_2014}, as well as
\autoref{lem:noinvariantmeasures}).
Not surprisingly,
the techniques used to establish $\mathcal{Z}$-stability of 
$C(X)\rtimes G$ are quite different depending on whether the 
group $G$ is amenable or not. 


On the amenable side, one of the first results in this 
direction is due to Toms and Winter \cite{TomWin_minimal_2009},
who showed that $C(X)\rtimes\Z$ is $\mathcal{Z}$-stable
whenever $\Z\curvearrowright X$ is free and minimal, and $\dim(X)<\infty$. 
The efforts to extend this result to
more general groups led Kerr to introduce 
the notion of \emph{almost finiteness} for topological 
actions of amenable groups in \cite{Ker_dimension_2020}, and prove that 
crossed products by free, minimal and almost finite actions 
are $\mathcal Z$-stable. 
Almost finiteness has been
verified in a number of cases of interest \cite{KerSza_almost_2020}, 
and the most
recent result in this setting is by 
Kerr and Naryshkin \cite{KerNar_elementary_2021}, 
who proved that free actions of elementary amenable groups on 
finite-dimensional spaces are automatically almost finite.
It is not possible to completely remove the finite-dimensionality 
assumption in these:
Giol and Kerr constructed in \cite{GioKer_subshifts_2010} 
a free,
minimal homeomorphism of an infinite dimensional space
$X$ such that $C(X)\rtimes\Z$ is not $\mathcal{Z}$-stable.
The dividing line regarding 
$\mathcal{Z}$-stability for crossed products by amenable 
groups is expected to be mean dimension zero (or the conjecturally
equivalent notion of the small boundary property), which is
weaker than finite-dimensionality of the space. 
In this direction, 
Elliott and Niu showed in \cite{EllNiu_minimal_2017} that 
$C(X)\rtimes\Z$ is $\mathcal{Z}$-stable whenever
$\Z\acton X$ is free and minimal, and has mean dimension zero; 
this was generalized
by Niu \cite{Niu_comparison_2019} to $\Z^d$-actions. 
It has been conjectured
by Phillips and Toms that the converse should also be true, 
and there have been some partial
results in this direction; see 
\cite{HirPhi_radius_2020}. 

Much less seems to be known in the nonamenable setting, 
although certain  
classes of actions have been successfully 
studied from this point of view.
For example, Laca and Spielberg proved in 
\cite{LacSpi_purely_1996} that crossed products by
minimal, topologically free, \emph{strong boundary} actions 
are purely infinite. 
As a consequence, for actions as above which are in 
addition amenable, 
such crossed products are nuclear and thus
$\mathcal{O}_\infty$-stable by Kirchberg's absorption theorem \cite[Theorem 3.15]{KirPhi_embedding_2000},
so they are in particular $\mathcal Z$-stable.
Simiar results were obtained independently
by Anantharaman-Delaroche in \cite{AD_purely_1997}.
In \cite{JolRob_simple_2000}, Jolissaint and Robertson 
proved analogous results for the larger class of 
\emph{$n$-filling actions}. 

A property that is key in the study of dynamical
systems is Kerr's notion of dynamical comparison. 
Given nonempty open sets $U,V\subseteq X$, we write $U\prec V$ if 
every closed subset of
$U$ admits a finite open cover whose elements can be transported via the 
group action to pairwise disjoint subsets of $V$ (see \autoref{df:DynSubeq}).
A system $G\acton X$ has dynamical comparison if $U\prec V$ whenever
$\mu(U)<\mu(V)$ for all $G$-invariant probability measures $\mu$.
Establishing dynamical comparison is a powerful tool for proving 
$\mathcal{Z}$-stability of crossed products, 
both in the amenable and in the nonamenable settings. 
For amenable groups, the small boundary property
implies almost finiteness (and thus $\mathcal{Z}$-stability) in the presence of
dynamical comparison; see \cite[Theorem~A]{KerSza_almost_2020}. 
As it turns out, dynamical comparison has been verified in many
interesting cases: for free actions of groups with subexponential growth on 
Cantor spaces in \cite{DowZha_comparison_2020},
and for arbitrary minimal actions of groups 
with polynomial growth in \cite{Nar_polynomial_2021}. The latter result 
gives a large class of groups for which
the small boundary property implies $\mathcal{Z}$-stability.
For amenable, minimal, topologically free actions of nonamenable groups,
Ma proved that comparison implies pure infiniteness of the crossed product
(see \cite{Ma_comparison_2020}, and see \autoref{thm:Ma} for
a simple proof). Not surprisingly, establishing dynamical 
comparison is often very challenging.



In this work, we prove that all amenable and minimal actions of a large class
of nonamenable groups automatically satisfy dynamical comparison.
As a consequence, for actions which are additionally topologically free, the 
crossed products are purely infinite (and thus satisfy the assumptions of the classification theorem).
The following is the main definition of this work. 

\begin{defintro}\label{df:introParTow}
Given $n\in\N$, we say that a countable group $G$ admits \emph{$n$-paradoxical towers}, 
if for every finite subset $D\subseteq G$ there are 
$A_1,\ldots,A_n\subseteq G$ and $g_1,\ldots,g_n\in G$ such that:
\begin{enumerate}
\item The sets $dA_i$, for $d\in D$ and $i=1,\ldots,n$, are 
pairwise disjoint.
\item $G=\bigcup_{i=1}^{n}g_iA_i$.
\end{enumerate}
We say that $G$ admits \emph{paradoxical towers} if it admits $n$-paradoxical
towers for some $n\in \N$.
\end{defintro}

It is easy to see that a group admitting paradoxical towers is 
necessarily nonamenable. Elementary methods allow one to show 
that the free group $\mathbb{F}_n$ admits paradoxical towers; 
see \autoref{prop:F2}, and see Theorem~\ref{thm:introExamples} for more examples.
There exist nonamenable groups which do not admit paradoxical towers, such as
$\mathbb{F}_2\times \Z$; see \autoref{eg:F2Z}. 

We show that \emph{every} amenable and minimal action of 
a group with paradoxical towers has dynamical comparison.
In fact, our methods allow 
us to deal with products of such groups with arbitrary groups;
see \autoref{thm:groupstospaces}. 

\begin{thmintro}\label{thm:introCrossedProd}
Let $H$ be a countable group with paradoxical towers, let $K$
be an arbitrary countable group, and set $G=H\times K$.
Then every amenable, minimal action 
$G\acton X$ on a compact metrizable space has dynamical comparison. 
If $G\acton X$ is moreover topologically free, then the crossed product
$C(X)\rtimes G$ is a Kirchberg algebra satisfying the UCT.
\end{thmintro}

By \cite[Theorem~6.11]{RorSie_purely_2012}, every nonamenable exact group 
admits a large family of minimal, amenable, topologically free actions on compact metrizable spaces. 

The above result shows an unexpected phenomenon in the nonamenable
setting: classifiability of $C(X)\rtimes G$ does not require 
finite dimensionality of $X$ or any version of mean dimension zero 
for actions of nonamenable groups. There is thus a genuine difference
between the amenable and the nonamenable cases.
For a nonamenable group $G$ not covered by Theorem~\ref{thm:introCrossedProd},
we do not know if a simple, nuclear crossed product of the form 
$C(X)\rtimes G$ can be finite, although we strongly suspect that this
is not the case\footnote{\autoref{lem:noinvariantmeasures} implies that such
crossed products are never stably finite, and Conjecture~\ref{cnj:MinAmenActComp} below 
predicts that such crossed products are always purely infinite.}. If $G$ 
contains $\mathbb{F}_2$,
we show in \autoref{prop:PropInf} that a simple, nuclear crossed product 
of the form $C(X)\rtimes G$
is always properly infinite.

We complement Theorem~\ref{thm:introCrossedProd} by proving 
that large classes of nonamenable groups admit paradoxical towers; 
see \autoref{sec:examples}. We summarize some of the results:

\begin{thmintro}\label{thm:introExamples}
The following classes of groups admit paradoxical
towers:
\begin{enumerate}
\item Acylindrically hyperbolic groups; 
see~\autoref{eg:acylindrical}. In particular, all nonamenable
hyperbolic groups and thus all nonabelian free grups.
\item Highly transitive faithful non-ascending HNN-extensions; see \autoref{prop:HNN}. In
particular, Baumslag-Solitar groups $BS(m,n)$ with $|m|,|n|>1$ and $|m|\neq |n|$; 
see \autoref{eg:BS}.
\item All free products $G\ast H$ of nontrivial groups with $|H|>2$; 
see \autoref{eg:AmFreeProd} for a larger class.
\item Lattices in a real connected semisimple Lie groups without 
compact factors and with finite center (such as 
$\operatorname{SL}_n(\Z)$ for $n\geq 3$); see \autoref{eg:lattice}.
\item $\widetilde{A}_2$-groups; see \autoref{eg:building}.
\item Discrete subgroups of isometries of a visibility manifolds 
with finite co\-vo\-lume; see \autoref{eg:manifold}.
\end{enumerate}
\end{thmintro}

After these results appeared on the arXiv, further examples of groups with paradoxical towers and purely infinite crossed products were obtained by Ma and Wang 
in~\cite{MaWang}. Moreover, some of the techniques developed here have also been
successfully used in the study of actions on simple C*-algebras; 
see \cite{GGKNV_tracial_2022}.

Based on the evidence provided in this work, we expect
that the conclusion of Theorem~\ref{thm:introCrossedProd} should hold for 
arbitrary nonamenable groups:

\begin{cnjintro}\label{cnj:MinAmenActComp}
Let $G$ be a countable nonamenable group and let $X$ be a compact
metrizable space. Then every amenable, minimal action $G\acton X$ 
has dynamical comparison. 
\end{cnjintro}

A positive solution to the above conjecture would imply that
crossed products by amenable, minimal and topologically free 
actions of nonamenable groups are \emph{always} classifiable. 
Our conjecture would also imply a strengthening of Kirchberg's
dichotomy for crossed products: if $C(X)\rtimes G$ is simple
and nuclear, then it is either stably finite (if and only if
$G$ is amenable) or purely infinite (if and only if $G$ is
nonamenable), \emph{regardless} of whether it is 
$\mathcal{Z}$-stable or not.

\vspace{.25cm}
\textbf{Acknowledgements:}
We would like to thank Sahana H. Balasubramanya, Yair Hartman, David Kerr, 
Mario Klisse, Xin Ma, Shintaro Nishika\-wa, Tron Omland, Mikael 
R\o rdam, Hannes Thiel, Federico Vigolo, and 
Stuart White for helpful comments and discussions. We would also like to thank the anonymous referee for helpful suggestions.

\section{Amenable actions and dynamical subequivalence}\label{sec:preliminaries}
%
%

In this section, we collect a number of elementary definitions and results 
that will
be needed in the rest of the work. 
All countable groups will be endowed with the discrete topology.
All measures on locally compact spaces are assumed to be regular Borel measures. 
If $G$ is a discrete group, we denote by
$\mathrm{Prob}(G)\subseteq \ell^1(G)$ 
the set of all probability measures on it.
If $\mu\in \mathrm{Prob}(G)$
and $g\in G$, we denote by $g\cdot\mu$ the probability measure given by
$(g\cdot\mu)(E)=\mu(g^{-1}E)$ for $g\in G$ and $E\subseteq G$.

The following definition, introduced by Anantharaman-Delaroche and
Renault in~\cite{ADRen_amenable_2001}, is standard by now.

\begin{df}\label{df:AmenAction}
An action $G\acton X$ of a countable group $G$ on a compact metrizable space $X$ is said to be \emph{amenable} if there exists a sequence $(\mu_n)_{n\in\N}$ of continuous maps $\mu_n\colon X\to \Prob(G)$ such that for all $g\in G$ we have
\[\sup_{x\in X}\|\mu_n(g\cdot x)-(g\cdot\mu_n)(x)\|_1\xrightarrow{n\to \infty}0.\]
 
\end{df}

Note that a countable group is amenable if and only if it acts amenably on the one point space. 
More generally, an action $G\acton X$ on a compact Hausdorff space is amenable if and 
only if $C(X)\rtimes_r G$ is nuclear; see \cite[Corollary~6.2.14, Theorem~3.3.7]{ADRen_amenable_2001}, in which case the full and reduced crossed products of
$G\curvearrowright X$ agree.
The following lemma is folklore, and we 
include the proof for the convenience of the reader. 

\begin{lma}\label{lem:noinvariantmeasures}
Let $G\acton X$ be an amenable action of a countable group on a compact metrizable space. Then $G$ is amenable if and only if there exists a $G$-invariant probability 
measure on $X$. 
\end{lma}
\begin{proof}
For the ``only if'' implication, assume that $G$ is amenable and 
fix a $G$-invariant mean $\phi\colon L^{\infty}(G)\to \C$. Let $\eta$ be any probability measure on $X$. The Poisson map $P_\eta\colon C(X)\to L^{\infty}(G)$ defined by
\[P_{\eta}(f)(g)=\int_X f(g\cdot x)\ d\eta(x)\]
for $f\in C(X)$ and $g\in G$, 
is a unital positive $G$-equivariant map. Then $\phi\circ P_\eta\colon C(X)\to \C$ is a $G$-invariant state giving rise to a $G$-invariant probability measure on $X$. 
For the ``if'' implication, let $(\mu_n)_{n\in\N}$ be as in \autoref{df:AmenAction}, 
and let 
$\nu$ be a $G$-invariant probability measure on $X$. For $n\in\N$, define 
$\rho_n\in \mathrm{Prob}(G)$ by
\[\rho_n(E)=\int_X\mu_n(x)(E)\ d\nu(x)\]
for all $E\subseteq G$.
Then $\|\rho_n-g\cdot \rho_n\|_1\xrightarrow{n\to \infty}0$ for all $g\in G$,
and thus $G$ is amenable. 
\end{proof}
%
\begin{rem}
Recall that any trace on $C(X)\rtimes G$ induces a $G$-invariant
probability measure on $X$ by restriction, and conversely any
such measure induces a trace on $C(X)\rtimes G$ via the canonical
conditional expectation $C(X)\rtimes G\to C(X)$.
It thus follows from \autoref{lem:noinvariantmeasures}
that a nuclear crossed product
$C(X)\rtimes G$ 
has a trace if and only if $G$ is amenable. 
\end{rem}

We need the notion of dynamical subequivalence
for tuples of sets, which is the partial order used to
define the type semigroup of a dynamical system.

\begin{df}\label{df:DynSubeq}
Let $G\acton X$ be an action of a discrete group on a 
compact Hausdorff space.
Let $U_1,\ldots,U_n,V_1,\ldots,V_m$ be open subsets of $X$.
We say that the family $(U_i)_{i=1}^n$ is \emph{dynamically
subequivalent} to $(V_j)_{j=1}^m$, and write 
$(U_i)_{i=1}^n\prec (V_j)_{j=1}^m$, if 
for any closed subsets $A_i\subseteq U_i$,
for $i=1,\ldots,n$,
there exist finite open covers $\mathcal{W}_{i}$ of $A_i$, 
elements $g_W^{(i)}\in G$ for $W\in\mathcal{W}_i$, 
and a partition 
\[\mathcal{C}_1\sqcup \ldots \sqcup \mathcal{C}_m=\big\{(i,W)\colon i=1,\ldots,n, W\in \mathcal{W}_i\big\},\]
such that, for each $j=1,\ldots,m$ 
the sets $g_W^{(i)}\cdot W$, for $(i,W)\in \mathcal{C}_j$, are 
pairwise disjoint and contained in $V_j$.
Given a nonnegative integer $r$, we write
$(U_i)_{i=1}^n\prec_r (V_j)_{j=1}^m$ if 
$(U_i)_{i=1}^n\prec (V_j)_{j=1,\ldots,m,k=1,\ldots,r+1}$. In other words, 
$(U_i)_{i=1}^n\prec_r (V_j)_{j=1}^m$ if the family $(U_i)_{i=1}^n$ is subequivalent to $r+1$ disjoint 
copies of the family $(V_j)_{j=1}^m$.
\end{df}

We will identify tuples containing one element with their unique
element, and will thus write $U\prec_r V$ instead of 
$(U)\prec_r (V)$. Note that this definition of dynamical $r$-subequivalence
for open sets agrees with Kerr's
\cite[Definition~3.1]{Ker_dimension_2020}. We record here the observation that 
$\prec$ is transitive.

\begin{lma}\label{lma:transitive}
Let $G\acton X$ be an action of a discrete group on a compact 
Hausdorff space, and let 
$U_1,\ldots, U_n, V_1,\ldots,V_m, W_1,\ldots,W_r\subseteq X$
be open sets satisfying
\[(U_i)_{i=1}^n \prec (V_{j})_{j=1}^m \ \ \mbox{ and } \ \
(V_{j})_{j=1}^m\prec (W_k)_{k=1}^r.
\]
Then $(U_i)_{i=1}^n\prec (W_k)_{k=1}^r$.
\end{lma}
\begin{proof}
Let $A_i\subseteq U_i$, for $i=1,\ldots,n$, be closed subsets. 
Using that $(U_i)_{i=1}^n \prec (V_{j})_{j=1}^m$, find open covers $\mathcal W_i$ of $A_i$ and elements $g_W^{(i)}\in G$, for $i=1,\dotsc,n$ and $W\in \mathcal W_i$, and a partition
	\[\mathcal C_1\sqcup\ldots\sqcup\mathcal C_m=\big\{(i,W)\colon i=1,\dotsc,n, W\in \mathcal W_i\big\}\]
such that for each $j=1,\dotsc,m$, the sets $g_W^{(i)}W$ for $(i,W)\in \mathcal C_j$ are pairwise disjoint and contained in $V_j$. By shrinking the elements of the open covers $\mathcal W_i$ if necessary, we can without loss of generality assume that for each $j=1,\dotsc,m$, the set
\[B_j\coloneqq \bigcup_{(i,W)\in \mathcal C_j}g_W^{(i)}\overline W\]
is contained in $V_j$ as well. Using that $(V_{j})_{j=1}^m\prec (W_k)_{k=1}^r$, 
find open covers $\mathcal Y_j$ of $B_j$ and elements $h_Y^{(j)}\in G$, for $j=1,\dotsc,m$ and $Y\in \mathcal Y_j$, and a partition 
\[\mathcal D_1\sqcup \ldots\sqcup \mathcal D_r=\{(j,Y)\colon j=1,\dotsc,m,Y\in \mathcal Y_j\}\]
such that for each $k=1,\dotsc,r$, the sets $h_Y^{(j)}Y$ for $(j,Y)\in \mathcal D_k$ are pairwise disjoint and contained in $W_k$. 
For $i=1,\dotsc,n, W\in \mathcal W_i, j=1,\dotsc,m$ with $(i,W)\in \mathcal C_j$, and $Y\in \mathcal Y_j$, we define an open set
	\[Z_{W,Y}\coloneqq W\cap \big(g_W^{(i)}\big)^{-1}(Y)\subseteq A_i \cap \big(g_W^{(i)}\big)^{-1}(B_j).\]
Then for each $i=1,\dotsc,n$, the family
	\[\mathcal Z_i\coloneqq\left\{Z_{W,Y}\colon W\in \mathcal W_i, Y\in \mathcal Y_j, (i,W)\in \mathcal C_j\right\}\]
is an open cover of $A_i$. For $k=1,\dotsc,r$, we define 
	\[\mathcal E_k\coloneqq \left\{\left(i, Z_{W,Y}\right)\colon W\in \mathcal W_i, Y\in \mathcal Y_j, (i,W)\in \mathcal C_j, (j,Y)\in \mathcal D_k\right\}.\]
Note that
	\[\mathcal E_1\sqcup \ldots\sqcup \mathcal E_r=\{(i,Z)\colon i=1,\dotsc, n, Z\in \mathcal Z_i\}.\]
For $i=1,\dotsc,n, j=1,\dotsc,m$ and $Z_{W,Y}\in \mathcal Z_i$ with $Y\in \mathcal Y_j$, set  
\[t_{{W,Y}}^{(i)}\coloneqq h_Y^{(j)}g_W^{(i)}.\]. 
For fixed $k=1,\dotsc,r$, it easily follows from the construction that the sets $t_{{W,Y}}^{(i)}Z_{W,Y}$ for $(i,Z_{W,Y})\in \mathcal E_k$ are pairwise disjoint 
and contained in $W_k$. This shows that $(U_i)_{i=1}^n\prec (W_k)_{k=1}^r$,
as desired.
\end{proof}

We will ultimately only be interested in comparing individual 
open sets, but the perspective 
using tuples
will be helpful in the proof of \autoref{thm:groupstospaces}, 
since it will allow us to decrease the 
number of colors we need to obtain comparison. 
The reason for this is that 
$U\prec (V_j)_{j=1}^m$ is a much stronger condition than
$U\prec_{m-1} \bigcup_{j=1}^m V_{j}$.
For instance, it follows from the previous lemma that
$U\prec (V_j)_{j=1}^m\prec_r W$ implies $U\prec_r W$, while
the direct argument using 
$\bigcup_{j=1}^{m}V_j$ instead of $(V_j)_{j=1}^{m}$
would only yield $U\prec_{(r+1)m-1} W$.



\begin{df}[{\cite[Definition~3.2]{Ker_dimension_2020}}]
Let $X$ be a compact space, and let $r$ be a nonnegative integer. An action $G\acton X$
of a discrete group $G$ 
is said to have \emph{dynamical $r$-comparison}, if for any two nonempty open subsets $U,V\subseteq X$ satisfying $\mu(U)<\mu(V)$ for all
$G$-invariant probability measures $\mu$ on $X$, we have $U\prec_r V$.

If $r=0$, we say that $G\acton X$ has \emph{dynamical comparison}.
\end{df}

\begin{figure}[h]\label{fig:comparison}
\centering
\includegraphics[width=0.75\textwidth]{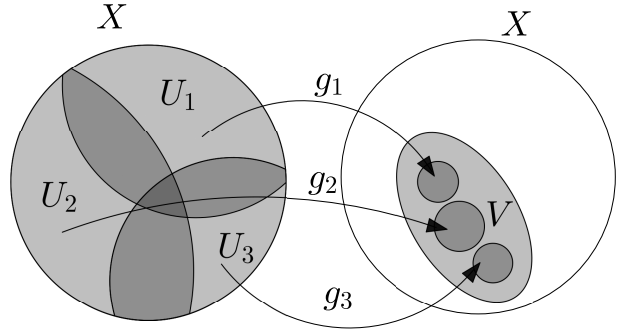} 
\caption{$X\prec V$}
\end{figure}

When $G\acton X$ is an amenable action of a nonamenable group, we have seen in
\autoref{lem:noinvariantmeasures} that there are no $G$-invariant probability measures 
on $X$. In particular, $G\acton X$ has $r$-dynamical comparison precisely
when $U\prec_r V$ for all nonempty open sets $U,V\subseteq X$. By transitivity, it suffices to check this for $U=X$, since every open set is subequivalent to the whole space.

While we will be interested in establishing dynamical comparison, the 
tools and arguments we use will only give dynamical $r$-comparison. 
For actions of nonamenable
groups without invariant probability measures, the following lemma shows that the 
two properties are in fact equivalent. This should be compared with 
\cite[Lemma~2.3]{Nar_polynomial_2021}, where it is shown that 
$r$-comparison implies comparison for minimal actions of amenable groups.




\begin{lma}\label{lma:minimalComparison}
Let $G\acton X$ be an action of a discrete group on a compact Hausdorff space with no 
invariant probability measures, 
and let $r$ be a nonnegative integer. Then $G\acton X$ has dynamical $r$-comparison if and only if $G\acton X$ has dynamical comparison.
\end{lma}
\begin{proof}
We prove the nontrivial implication, so assume that $G\acton X$ has dynamical $r$-comparison.
One readily shows, arguing as in the discussion after 
\cite[Definition 2.4]{Ma_comparison_2020}, 
that $G\curvearrowright X$ is minimal and $X$ has no isolated points.
As explained above, it suffices to fix a nonempty open set $V\subseteq X$
and show that $X\prec V$.
Fix $x\in V$, and note that $V\cap G\cdot x$ is an infinite set. Find $t_1,\ldots, t_{r+1}\in G$ such that $t_k\cdot x\in V$ for all $k=1,\ldots,r+1$ and 
$t_k\cdot x \neq t_\ell\cdot x$ whenever $k\neq \ell$. Using that $X$ is 
Hausdorff, find an open set $W\subseteq X$ such that $x\in W$ and $\{t_k\cdot W\}_{k=1}^{r+1}$ are pairwise disjoint sets in $V$. Since $X\prec_r W$ by assumption, 
there exist a finite open cover 
$\mathcal{O}=\mathcal{O}_1\sqcup\ldots\sqcup \mathcal{O}_{r+1}$ of $X$, and 
$g_{O}\in G$ for $O\in\mathcal{O}$, such that the sets $g_{O}\cdot O$, for 
$O\in \mathcal{O}_k$ are pairwise disjoint subsets of $W$, for every $k=1,\ldots, r+1$.
Now, $\{t_kg_{O}\cdot O\}_{O\in\mathcal{O}_k,k=1,\ldots, r+1}$ is a collection of pairwise disjoint sets in $V$, verifying that $X\prec V$ as desired.
\end{proof}

We close this section by giving a simple proof of Theorem~1.1 
from~\cite{Ma_comparison_2020}, which avoids the use of scaling
elements and hereditary subalgebras in favor of Cuntz semigroup techniques.
(We refer the reader to \cite[Chapter 2]{AntPerThi_tensor_2018} or 
\cite[Sections~2 and~3]{GarPer_modern_2022} for an 
introduction to these.)
Given positive elements $a$ and $b$ in a $C^*$-algebra $A$, we say that $a$ is 
\emph{Cuntz subequivalent} to $b$ in $A$, written 
$a\precsim b$ in $A$, if there exists a sequence $(c_n)_{n\in\N}$ in $A$ such that
$\lim_{n\to\infty}c_n^*bc_n= a$. We write $a\sim b$ if $a\precsim b$ 
and $b\precsim a$. We will use the fact that if $A$ is abelian, 
then $a\precsim b$ if and only if 
the open support of $a$ is contained in that of $b$. In particular, in a general
$C^*$-algebra, if 
$a$ and $b$ commute and $0\leq a\leq b$, then $a\precsim b$.

\begin{thm}[Ma]\label{thm:Ma}
Let $G\acton X$ be a minimal and topologically free action of a discrete group 
on a compact Hausdorff space. Assume that $G\acton X$ has dynamical comparison 
and admits no invariant probability measures. Then $C(X)\rtimes_r G$ is simple
and 
purely infinite.
\end{thm}
\begin{proof} Simplicity follows from \cite[Theorem~2]{ArcSpi_topologically_1994}, 
so we prove pure infiniteness. 
Let $a,b\in C(X)\rtimes_r G$ be nonzero positive contractions. We will show that 
$b\precsim a$ in $C(X)\rtimes_r G$, which implies that
$C(X)\rtimes_rG$ is purely infinite
by \cite[Defintion~4.1]{KirRor_nonsimple_2000}. Since $b\precsim 1_{C(X)}$, it is enough to show that $1_{C(X)}\precsim a$ in $C(X)\rtimes_r G$.
By \cite[Lemma~7.9]{Phi_large_2014}, 
there exists a nonzero positive contraction $f\in C(X)$ such that 
$f\precsim a$ in $C(X)\rtimes_rG$. Set $W=\{x\in X\colon f(x)>0\}$, and
choose $U$ to be a nonempty open subset of $W$ such that $\overline{U}\subseteq W$.
There exists a positive contraction $g\in C(X)$ such that $g=0$ on $X\setminus W$ and $g=1$ on $\overline{U}$.
Since $X\prec U$,
\cite[Lemma~12.3]{Ker_dimension_2020} implies that
$1_{C(X)}\precsim g$ in $C(X)\rtimes_rG$.
On the other hand, we have 
$g\precsim f$ since $\{x\in X\colon g(x)>0\}\subseteq W=\{x\in X\colon f(x)>0\}$. Transitivity of Cuntz subequivalence gives $1_{C(X)}\precsim f$, and so $1_{C(X)}\precsim a$, as desired.
\end{proof}

\section{Paradoxical towers give dynamical comparison}\label{sec:paradoxical towers}

In this section, we introduce the notion of \emph{paradoxical 
towers}, which is the main technical tool in this work. 
For (certain extensions of) 
groups admitting paradoxical towers, we show that amenable, minimal 
actions always have dynamical comparison; see \autoref{thm:groupstospaces}. 
Using this, we establish
classifiability of a large class of crossed products 
in \autoref{cor:PurInf}. 
Examples are discussed in \autoref{sec:examples}.

\begin{df}\label{df:ParadoxicalTower}
Let $n\in\mathbb{N}$. We say that a countable group $G$ 
admits \emph{$n$-paradoxical towers} if for every finite subset $D\subseteq G$ there are $A_1,\ldots,A_n\subseteq G$ and $g_1,\ldots,g_n\in G$ such that 
\begin{enumerate}
\item the sets $dA_i$, for $d\in D$ and $i=1,\ldots,n$, are pairwise disjoint.
\item $G=\bigcup_{i=1}^{n}g_iA_i$.
\end{enumerate}
%
We say that $G$ admits \emph{paradoxical towers} if there is $n\in\N$ such that $G$ admits $n$-paradoxical towers.
\end{df}

It is easy to see that a group admitting paradoxical towers is 
necessarily nonamenable. 
The class of groups admitting paradoxical towers is very 
large, but it does not exhaust all nonamenable groups; for example,
$\mathbb{F}_2\times\Z$ does not admit paradoxical towers (see \autoref{eg:F2Z}).
We postpone this discussion until \autoref{sec:examples}, 
and we only present here the
following basic example (see \autoref{eg:acylindrical} for a
much larger class).

\begin{prop}\label{prop:F2}
The free group $\mathbb{F}_2$ admits 2-paradoxical towers. In 
fact, given a finite subset $D\subseteq \mathbb{F}_2$ there are nonempty subsets
$A_1,A_2,A_3\subseteq \mathbb{F}_2$ and $g_1,g_2,g_3\in \mathbb{F}_2$
such that: 
\be\item the sets $dA_j$, for $d\in D$ and $j=1,2,3$, are pairwise disjoint,
\item the sets $\mathbb{F}_2\setminus g_jA_j$, for $j=1,2,3$, are pairwise
disjoint.
\ee
\end{prop}
\begin{proof}
We begin by observing that the property in the statement implies that 
$\mathbb{F}_2$ has 2-paradoxical towers. Indeed, condition (2)
implies that 
\[\emptyset=\big(\mathbb{F}_2\setminus g_1A_1\big)\cap \big(\mathbb{F}_2\setminus g_2A_2\big),\]
and by taking complements we get $\mathbb{F}_2=g_1A_1\cup g_2A_2$. In particular,
$A_1,A_2$ and $g_1,g_2$ satisfy the conditions of \autoref{df:ParadoxicalTower} for $n=2$.

Denote by $a,b$ the generators of $\mathbb{F}_2$, and set 
$L=\{a,b,a^{-1},b^{-1}\}$. For $h\in \mathbb{F}_2$, we write
$W(h)\subseteq \mathbb{F}_2$ for the set of all reduced words 
with letters from $L$ which start with $h$. If
$x\in L$ is the last letter of $h$, then a generic element of $W(h)$ has the
form $hg$ with $g\notin W(x^{-1})$. In particular,
\begin{equation}\label{eqn:ComplementF2}\tag{3.1}
\mathbb{F}_2\setminus h^{-1}W(h)= W(x^{-1}).
\end{equation}
For $r\geq 0$, we write $B_r$ for the set 
of all reduced words of length at most $r$; note that $B_rB_s= B_{r+s}$ for 
all $r,s>0$.
Let $D\subseteq \mathbb{F}_2$ be a finite set.
Find $m\geq 0$ with $D\subseteq B_m$, and define 
\[h_1=a^{2m}ba, \ \ h_2=a^{2m}ba^{-1}, \ \ \mbox{ and }
\ \ h_3=a^{2m}b^2.
\]
For $i=1,2,3$, set $A_i=W(h_i)$. 
We claim that these sets satisfy condition (1).
Since every element of $D$ has length at most $m$, it suffices to check
that whenever $e\in B_{2m}$ and $i,j=1,2,3$ satisfy 
$eA_i\cap A_{j}\neq \emptyset$, then $e=1$ and $i=j$.
Given $e,i,j$ as above, if $eA_i$ intersects $A_j$ then there is 
$g\in A_i$ such that $eg$, after reduction, starts with $2m$ copies of $a$. 
Since the $(2m+1)$-st letter of $g$ is not an $a$, it follows that 
the product
of $e$ and $g$ cannot have any cancelations, and thus
$e=a^{k}$ for some $0\leq k\leq m$. 
On the other hand, $a^{k}h_i$ never has $h_j$ as an initial segment unless $k=0$
and $i=j$. This proves the claim.

For $i=1,2,3$, set $g_i=h_i^{-1}$. Using \eqref{eqn:ComplementF2}, we get 
\[\mathbb{F}_2\setminus g_1A_1=W(a^{-1}), \ \ 
 \mathbb{F}_2\setminus g_2A_2=W(a), \ \ \mbox{ and } \ \
 \mathbb{F}_2\setminus g_3A_3=W(b^{-1}), \ \
\]
and these sets are clearly pairwise disjoint. This finishes the proof.
\end{proof}

We will need some auxiliary lemmas. 
In the following, for $m=1$ we just get the definition of
paradoxical towers. In general, the strengthening is that
the towers $A_i^{(j)}$ are jointly (and not just separately) 
$D$-free.

\begin{lma}\label{lma:MoreTowers}
Let $n\in\N$ and let $G$ be a countable group with $n$-paradoxical
towers, and let $D\subseteq G$ be a finite subset.
For every $m\in\N$, there exist
subsets $A_i^{(j)}\subseteq G$ and group elements $g_i^{(j)}\in G$,
for $i=1,\ldots,n$ and $j=1,\ldots,m$, such that:
\begin{enumerate}
\item the sets $dA^{(j)}_i$, for $d\in D$, $i=1,\ldots,n$,
and $j=1,\ldots,m$, are pairwise disjoint.
\item $G=\bigcup_{i=1}^{n}g^{(j)}_iA^{(j)}_i$ for every $j=1,\ldots,m$.
\end{enumerate}
\end{lma}
\begin{proof}
Let $D\subseteq G$ be a finite subset.
Since $G$ is infinite, there exist $s_1,\ldots,s_m\in G$
such that $Ds_j$, for $j=1,\ldots,m$, are pairwise disjoint sets.
Set $\widetilde{D}=\bigsqcup_{j=1}^m Ds_j$,
which is a finite subset of $G$.
Since $G$ admits $n$-paradoxical towers, there are sets 
$A_1,\dotsc,A_n\subseteq G$ and elements $g_1,\dotsc,g_n\in G$ such that 
the sets $\widetilde{d} 
{A}_i$, for $\widetilde{d}\in 
\widetilde{D}$ and $i=1,\dotsc,n$, are pairwise disjoint, and
$G=\bigcup_{i=1}^n{g}_i{A}_i$.

For $i=1,\ldots,n$ and $j=1,\ldots,m$, set 
\[A_i^{(j)}=s_jA_i \ \ \mbox{ and } \ \ g_i^{(j)}=g_is_j^{-1}.\]
One readily checks that conditions (1) and (2) in the statement are
satisfied.
\end{proof}



Given a metric space $(X,d)$, a set $U\subseteq X$ and $\varepsilon>0$, we set
\[U^{-\varepsilon}= \{x\in U\colon d(x,X\setminus U)>\varepsilon \}. \]


\begin{lma}\label{lma:Petr}
Let $G\acton X$ be an action of a countable group
on a compact metric space $X$,
let $n$ be a nonnegative integer, let $\varepsilon>0$, 
and let $D\subseteq G$ be a finite symmetric set.
\be\item 
Let $V, U_0,\ldots,U_n\subseteq X$ be open sets and let 
$R\colon X\to [0,\infty)$ be a function satisfying
\[\left|\{g\in D^2\colon g\cdot x\in V\}\right|+R(x)<\sum_{k=0}^n\left|\{g\in D\colon g\cdot x\in U_k^{-\varepsilon}\}\right|,\]
for all $x\in X$. Then for every closed subset
$A\subseteq V$ there exist $0<\widetilde{\varepsilon}<\varepsilon$, 
a finite open cover
$\mathcal{O}$ of $A$, 
group elements $s_O\in G$, for $O\in\mathcal{O}$, and a partition
$\mathcal{O}=\mathcal{O}_0\sqcup \ldots\sqcup \mathcal{O}_n$
satisfying the following properties:
\be\item 
for every $k=0,\ldots,n$, 
the family $\{s_O\cdot O\colon O\in\mathcal{O}_k\}$
consists of pairwise disjoint subsets of $U_k$, 
\item with $B_k$ denoting the 
closure of $\bigcup_{O\in\mathcal{O}_k} s_O\cdot O$
and $\widetilde{U}_k=U_k\setminus B_k$,
we have
\[R(x)<\sum_{k=0}^n\left|\{g\in D\colon g\cdot x\in \widetilde{U}_k^{-\widetilde{\varepsilon}}\}\right|, \]
for all $x\in X$. 
\ee
\item Let $V_1,\ldots,V_m,U\subseteq X$ be open sets
satisfying 
\[\sum_{j=1}^m \left|\{g\in D^2\colon g\cdot x\in V_j\}\right|<(n+1)\left|\{g\in D\colon g\cdot x\in U^{-\varepsilon}\}\right|, \]
for all $x\in X$. Then $(V_j)_{j=1}^m\prec_{n}U$.
\ee
\end{lma}
\begin{proof}
(1). This is proved exactly as 
\cite[Lemma 3.1]{Nar_polynomial_2021}. We omit the proof.

(2). We will prove this by repeatedly applying part (1).
For each $j=1,\ldots,m$, let $A_j\subseteq V_j$ be a closed subset.
Set $\varepsilon^{(1)}=\varepsilon$ and 
$U^{(1)}_0=\ldots=U^{(1)}_n=U$.
For $j=1,\ldots,m$, let $R^{(j)}\colon X\to [0,\infty)$ be given by
$R^{(j)}(x)=\sum_{i=j}^m \left|\{g\in D^2\colon g\cdot x\in V_i\}\right|$
for $x\in X$. By construction, we have
\[\left|\{g\in D^2\colon g\cdot x\in V_1\}\right|+R^{(1)}(x)<\sum_{k=0}^n\left|\big\{g\in D\colon g\cdot x\in (U^{(1)}_k)^{-\varepsilon^{(1)}}\big\}\right|\]
for all $x\in X$. By part (1), there exist $0<\varepsilon^{(2)}<\varepsilon^{(1)}$, an open cover
$\mathcal{O}^{(1)}$ of $A_1$, group elements $s_O^{(1)}\in G$,
for $O\in\mathcal{O}^{(1)}$, and a partition
$\mathcal{O}^{(1)}=\mathcal{O}^{(1)}_0\sqcup \ldots\sqcup \mathcal{O}^{(1)}_n$
satisfying the following properties:
\be\item[(a.1)] 
for every $k=0,\ldots,n$, 
the family $\{s_O\cdot O\colon O\in\mathcal{O}^{(1)}_k\}$
consists of pairwise disjoint subsets of $U^{(1)}_k$, 
\item[(b.1)] with $B^{(1)}_k$ denoting the 
closure of $\bigcup_{O\in\mathcal{O}^{(1)}_k} s_O\cdot O$
and $U^{(2)}_k=U^{(1)}_k\setminus B_k\subseteq U$,
we have
\[R^{(1)}(x)<\sum_{k=0}^n\left|\big\{g\in D\colon g\cdot x\in (U^{(2)}_k)^{-\varepsilon^{(2)}}\big\}\right|, \]
for all $x\in X$. 
\ee

By construction, we get 
\[\left|\{g\in D^2\colon g\cdot x\in V_2\}\right|+R^{(2)}(x)<\sum_{k=0}^n\left|\big\{g\in D\colon g\cdot x\in (U^{(2)}_k)^{-\varepsilon^{(2)}}\big\}\right|,\]
for all $x\in X$. One continues applying part~(1) inductively.
After $m$ steps, 
we will have constructed, for each $j=1,\ldots,m$, 
an open cover
$\mathcal{O}^{(j)}$ of $A_j$, group elements $s_O^{(j)}\in G$,
for $O\in\mathcal{O}^{(j)}$, and a partition
$\mathcal{O}^{(j)}=\mathcal{O}^{(j)}_0\sqcup \ldots\sqcup \mathcal{O}^{(j)}_n$ such that for every $k=0,\ldots,n$, 
the family $\{s_O\cdot O\colon O\in\mathcal{O}^{(j)}_k\}$
consists of pairwise disjoint subsets of $U^{(j)}_k\subseteq U$.
For $j=1,\ldots,m$, set
\[\mathcal{C}_j=\{(k,O)\colon k=1,\ldots,n, O\in\mathcal{O}^{(j)}_k\},\]
and note that
\[\mathcal{C}_1\sqcup \cdots \sqcup \mathcal{C}_n=\{(k,O)\colon k=1,\ldots,n, O\in\mathcal{O}^{(j)}_k, j=1,\ldots,m\}.\]

Since the sets $U^{(j)}_k$, for $j=1,\ldots,m$ are pairwise disjoint subsets of $U$,
it follows that the above choices
witness the fact that $(V_j)_{j=1}^m\prec_n U$, as desired.
\end{proof}


In the next lemma, note that we can not demand that the 
sets $eB_j$ for $e\in E$ \emph{and} $j=1,\ldots,m$ be 
pairwise disjoint, as this cannot happen if $K$
is amenable. 

\begin{lma}\label{lem:Cayley}
Let $K$ be a countable group and let $E\subseteq K$ be a finite symmetric subset containing the unit of $K$. 
Set $m=|E^2|$. Then there is a finite partition
\[K=B_1\sqcup \ldots\sqcup B_m, \]
such that for each $j=1,\ldots, m$, the sets $eB_j$, for $e\in E$, are pairwise disjoint.
\end{lma}
\begin{proof}
Consider the Cayley graph $\mathcal{G}=\mathrm{Cay}(K,E^2)$ whose vertices are the elements of $K$ and whose edges are of the form $(k,gk)$, for $k\in K$ and $g\in E^2\setminus\{1\}$.
Note that every vertex in $\mathcal{G}$ has exactly
$m-1$ edges coming out of it, and that there are no loops 
in $\mathcal{G}$.
The greedy coloring algorithm then implies that we can color 
the vertices of $\mathcal{G}$ using at most $m$ colors, 
in such a way that every two adjacent vertices have different 
colors\footnote{One way to do this is to enumerate the vertices and then color them inductively.}.
For $j=1,\ldots,m$, let $B_j\subseteq K$ denote the vertices 
with the $j$-th color. Then $B_1\sqcup \cdots\sqcup B_m=K$. Moreover, 
for $j=1,\ldots,m$, we have $gB_j\cap B_j=\emptyset$ for all 
$g\in E^2\setminus\{1\}$. Thus the sets $eB_j$, for $e\in E$, 
are pairwise disjoint for each $j=1,\ldots,m$.
\end{proof}




The following is the main result of this work. The main consequence
is the classifiability of the associated crossed products;
see \autoref{cor:PurInf}.
In its proof, we will work with doubly-indexed sets $V_{i,j}$ for 
$i=1,\ldots,n$ and $j=1,\ldots,m$. To lighten the notation, we will
write $(V_{i,j})_{i,j=1}^{n,m}$ for the tuple 
$(V_{i,j})_{i=1,\ldots,n, j=1,\ldots,m}$, and similarly for their
union $\bigcup_{i,j=1}^{n,m}$, or for sums $\sum_{i,j=1}^{n,m}$ 
indexed both by $i$ and $j$.

\begin{thm}\label{thm:groupstospaces}
Let $H$ be a countable group admitting paradoxical towers, let $K$ be any
countable group, and set $G=H\times K$. 
Then any amenable, minimal action of $G$ on a compact metrizable
space has dynamical comparison. 
\end{thm}
\begin{proof}

Let $X$ be a compact metrizable space and let 
$G\acton X$ be an amenable, minimal action. 
Let $n\in\mathbb{N}$ be such that $H$ admits $n$-paradoxical towers. By \autoref{lem:noinvariantmeasures} and \autoref{lma:minimalComparison}, 
and since $G$ is nonamenable, it suffices to show that $G\acton X$ has 
dynamical $n$-comparison. Let $U\subseteq X$ be a nonempty open set. 
Fix a metric on $X$ inducing its topology, and 
choose $\varepsilon>0$ such that $U^{-\varepsilon}\not=\emptyset$. By 
minimality of $G\acton X$, there is a finite set $F_0\subseteq G$
such that $F_0^{-1}\cdot U^{-\varepsilon}=X$. Without loss of generality, we
assume that $F_0$ contains the unit of $G$, and has the form $F_0=D_0\times E$ 
for finite sets $D_0\subseteq H$ and $E\subseteq K$ with $E=E^{-1}$.
Set $m=|E^4|$.

Since $H$ is infinite, we can find 
$t_1,\ldots,t_m \in H$ such that
the sets $D_0t_j$, for $j=1,\ldots,m$, are pairwise
disjoint. 
Let $D$ be any finite symmetric subset of $H$
containing $\bigsqcup_{j=1}^mD_0t_j$, 
set $F=D\times E$, which is finite and symmetric. 
With $s_j=(t_j,1)\in G$, 
note that the sets $F_0s_j $ are pairwise disjoint
and contained in $F$.
\vspace{.1cm}

\textbf{Claim~1:} \emph{for all $x\in X$, we have}
\begin{equation}\label{eq:coverUktimes}\tag{3.2}
\left|\left\{g\in F\colon g\cdot x\in U^{-\varepsilon}\right\}\right|\geq m.
\end{equation}

To prove the claim, fix $x\in X$ and $j=1,\ldots,m$. 
Denote by $F_x$ the set in the left-hand side of the displayed equation above.
Since 
\[s_j^{-1}\underbrace{F_0^{-1}\cdot U^{-\varepsilon}}_{=X}=X,\] 
there is $f_j\in F_0$ such that $f_js_j\cdot x\in U^{-\varepsilon}$, 
and thus $f_js_j$ belongs to $F_x$ (in addition to $F_0s_j$). 
Hence $|F_x\cap F_0s_j|\geq 1$ for all $j=1,\ldots,m$, 
and since the sets $F_0s_j$ are pairwise
disjoint, this shows that $|F_x|\geq m$, as desired.

\vspace{.1cm}

Since $H$ admits $n$-paradoxical towers, use \autoref{lma:MoreTowers}, 
to find
$A_i^{(j)}\subseteq H$ and $h_i^{(j)}\in H$, for 
$i=1,\ldots,n$ and $j=1,\ldots,m$
satisfying:
\begin{enumerate}
\item[(a.1)] the sets $dA^{(j)}_i$, for $d\in D^2$, $i=1,\ldots,n$,
and $j=1,\ldots,m$, are pairwise disjoint.
\item[(a.2)] $\bigcup_{i=1}^{n}h^{(j)}_iA^{(j)}_i=H$ for every $j=1,\ldots,m$.
\end{enumerate}
Use \autoref{lem:Cayley}, with $E^2$ in place of $E$, to find subsets 
$B_1,\ldots,B_m\subseteq K$ such that 
\be\item[(b.1)] for each $j=1,\dotsc,m$, the sets $e B_j$, for $e\in E^2$, 
are pairwise disjoint.
\item[(b.2)] $K=B_1\sqcup\ldots\sqcup B_m$. \ee

For $i=1,\ldots,n$ and $j=1,\ldots,m$, set 
\[C_{i,j}=  A_i^{(j)}\times B_j \subseteq G \ \ \mbox{ and } \ \
g_{i,j}=(h_i^{(j)},1)\in G.\]
We proceed to show the following: 
\begin{enumerate}
\item[(i)] The sets $fC_{i,j}$ for $f\in F^2$, $i=1,\dotsc,n$, 
and $j=1,\dotsc,m$, are pairwise disjoint. 
\item[(ii)] $\bigcup_{i,j=1}^{n,m} g_{i,j}C_{i,j}=G$.
\end{enumerate}

To check (i), let $i,i'=1,\ldots,n$, let $j,j'=1,\ldots,m$, 
and let $f,f'\in F^2$. Write $\pi_H\colon G\to H$ for the projection
onto the first coordinate, and $\pi_K\colon G\to K$ for the projection
onto the second one. Assume that 
\begin{equation}\label{eq:Ffree}\tag{3.3}
fC_{i,j}\cap f'C_{i',j'}\neq \emptyset.
\end{equation}
Apply $\pi_H$ to \eqref{eq:Ffree} to get 
$\pi_H(f)A_i^{(j)}\cap \pi_H(f')A_{i'}^{(j')}\neq\emptyset$. 
Since $\pi_H(F^2)= D^2$, condition (a.1) above implies 
that $i=i'$, $j=j'$ and $\pi_H(f)=\pi_H(f')$. Applying $\pi_K$ to 
\eqref{eq:Ffree} now gives 
$\pi_K(f)B_j\cap \pi_K(f')B_j\neq \emptyset$. 
Since $\pi_K(F^2)=E^2$, it thus follows from (b.1) above that 
$\pi_K(f)=\pi_K(f')$ and thus $f=f'$. This proves (i).
Part (ii) is immediate from (a.2) and (b.2).
\vspace{.2cm}

Now fix $0<\delta <(2nm(nm+1))^{-1}$. Use amenability of 
$G\acton X$ to find a continuous map 
$\mu\colon X\to \Prob(G)$ satisfying
\begin{equation}\label{eq:almostinvariant}\tag{3.4}
\sup_{x\in X}\|\mu(g\cdot x)-g\cdot \mu(x)\|_1<\delta
\end{equation}
for all $g\in F^2\cup\{g_{i,j}\}_{i,j=1}^{n,m}$.
For $i=1,\ldots,n$ and $j=1,\dotsc,m$, set
\[V_{i,j}= \Big\lbrace x\in X\colon \mu(x)(C_{i,j})>\frac{1}{nm+1} +\delta\Big\rbrace, \]
and note that $V_{i,j}$ is an open subset of $X$.
\vspace{.1cm}

\textbf{Claim~2:} \emph{we have $X\prec (V_{i,j})_{i,j=1}^{n,m}$.}
For $i=1,\ldots,n$ and 
$j=1,\ldots,m$, define
\[W_{i,j}=\left\lbrace x\in X\colon \mu(x)(g_{i,j}C_{i,j})>\frac{1}{nm+1} +2\delta\right\rbrace.\]
Then $W_{i,j}$ is open in $X$. Fix $x\in X$.
Since $\mu(x)$ is a probability measure
on $G$, by condition (ii) there are $i_x\in \{1,\ldots,n\}$ and 
$j_x\in \{1,\ldots,m\}$ such that 
\[\mu(x)(g_{i_x,j_x}C_{i_x,j_x})\geq \frac{1}{nm}>\frac{1}{nm+1}+2\delta.\] 
In other words, $x\in W_{i_x,j_x}$, and thus $X=\bigcup_{i,j=1}^{n,m} W_{i,j}$. 
Using \eqref{eq:almostinvariant} again, we get
\[g_{i,j}^{-1}W_{i,j}\subseteq \left\lbrace x\in X\colon \mu(x)( C_{i,j})>\frac{1}{nm+1}+\delta\right\rbrace=V_{i,j}\]
for $i=1,\ldots,n$ and $j=1,\dotsc,m$.
This shows that $X\prec (V_{i,j})_{i,j=1}^{n,m}$, as desired.
\vspace{.1cm}

\textbf{Claim~3:} \emph{we have $(V_{i,j})_{i,j=1}^{n,m}\prec_n U$.}
To prove the claim, note that 
by \eqref{eq:almostinvariant} and the fact that $F=D\times E$ is
symmetric, we have 
\[fV_{i,j}\subseteq \left\lbrace x\in X\colon \mu(x)(fC_{i,j})>\frac{1}{nm+1}\right\rbrace\]
for all $f\in F^2$, $i=1,\ldots,n$, and $j=1,\dotsc,m$.
Since the sets $fC_{i,j}$ are pairwise disjoint by
condition (i) above, for any $x\in X$ at most $nm$ of 
them can have $\mu(x)$-measure more than 
$\frac{1}{nm+1}$. We deduce that each $x\in X$
belongs to at most $nm$ of the sets $fV_{i,j}$, 
for $f\in F^2$, $i=1,\ldots,n$ and $j=1,\ldots,m$.
That is, for all $x\in X$, we get
\[\sum_{i,j=1}^{n,m}\left|\{f\in F^2\colon f\cdot x\in V_{i,j}\}\right|\leq nm \stackrel{\eqref{eq:coverUktimes}}{<}(n+1)\left|\{f\in F\colon f\cdot x\in U^{-\varepsilon}\}\right|. \]
By part~(2) of \autoref{lma:Petr}, we conclude that 
$(V_{i,j})_{i,j=1}^{n,m}\prec_n U$. 
\vspace{.1cm}

Combining Claims 2 and 3, we get $X\prec (V_{i,j})_{i,j=1}^{n,m}\prec_n U$,
which implies $X\prec_n U$ by
\autoref{lma:transitive}. This concludes the proof. 
\end{proof}

\begin{rem}
The above proof does not show that $H\times K$ admits paradoxical towers. Indeed, 
although the sets $C_{i,j}$ satisfy the conditions for $nm$-paradoxical towers, 
the number $m$ depends on the finite subset $E\subseteq K$. In fact, one can show 
that $G=H\times K$ never has paradoxical towers if
$K$ is infinite and amenable; see \autoref{eg:F2Z}.
\end{rem}

By \cite[Theorem~6.11]{RorSie_purely_2012}, 
every exact nonamenable group admits a large family of amenable, 
minimal, free actions on compact metric spaces. In particular, 
actions satisfying the assumptions of \autoref{thm:groupstospaces}
always exist. 

We obtain the following corollary:

\begin{cor}\label{cor:PurInf}
Let $H$ be a group admitting paradoxical towers, let $K$ be any countable group, and 
set $G=H\times K$. Let $G\acton X$ be an amenable, minimal and topologically free action on a compact metrizable space $X$. Then the crossed product $C(X)\rtimes G$ is a Kirchberg
algebra satisfying the UCT.
\end{cor}
\begin{proof} 
It is well known that $C(X)\rtimes G$ is simple, separable, unital and nuclear, 
and it satisfies the UCT by \cite[Theorem~10.9]{Tu_conjecture_1999}. Finally, it is purely infinite
by the combination of \autoref{thm:groupstospaces} and \autoref{thm:Ma}.
\end{proof}

\autoref{cor:PurInf} reveals an unexpected phenomenon 
in the nonamenable
setting: classifiability of $C(X)\rtimes G$, for the 
groups $G$ to which the corollary applies, 
does not require any finite dimensionality assumption 
on $X$, or any version of mean dimension zero. 
There is thus a genuine difference between the 
amenable and the nonamenable case.

It is an interesting and challenging problem to compute the possible
$K$-groups of Kirchberg algebras arising as in \autoref{cor:PurInf}.
Some progress in this direction has been made in \cite{EllSie_ktheory_2016}.

For a nonamenable group $G$, 
a simple, nuclear crossed 
product of the form $C(X)\rtimes G$ 
cannot be stably finite by \autoref{lem:noinvariantmeasures},
but we do not know if it can ever be finite if $G$ 
is not covered by \autoref{cor:PurInf} (although 
Conjecture~\ref{cnj:MinAmenActComp} predicts that this can never happen).
Using a weak version of paradoxical towers, we show below 
that if $G$ contains $\mathbb{F}_2$, then a 
simple, nuclear crossed product $C(X)\rtimes G$ 
is automatically properly infinite; see \autoref{prop:PropInf}.
Recall (see \cite[Definition~1.1]{Ror_classification_2002}) 
that a unital $C^*$-algebra $A$ is 
\emph{properly infinite} if there exist two mutually
orthogonal projections in $A$, 
each of which is Murray-von Neumann 
equivalent to the unit.

\begin{thm}\label{prop:PropInf}
Let $G$ be a countable group containing a nonabelian free group.
If $G\acton X$ is an amenable, minimal and topologically
free action on a compact metrizable space, then 
$C(X)\rtimes G$ is a properly infinite, simple, separable,
nuclear, unital $C^*$-algebra.
\end{thm}
\begin{proof}
We only need to show that $C(X)\rtimes G$ is properly
infinite.
By \cite[Proposition~2.2]{Cun_structure_1977}, 
it suffices to show that there is an isometry
in $C(X)\rtimes G$ which is not a unitary. 
Let $H\subseteq G$ be a nonabelian free subgroup of rank 2.
Let $h\in H\setminus\{1\}$ and set $D=\{1,h\}$.
By \autoref{prop:F2},
there exist nonempty sets
$B_1,B_2,B_3\subseteq H$ and $h_1,h_2,h_3\in H$ such 
that 
\begin{enumerate}
\item[(1)] the sets $dB_j$, for $d\in D$ and $j=1,2,3$, are pairwise disjoint, and
\item[(2)] the sets $H\setminus h_jB_j$, for $j=1,2,3$, are pairwise
disjoint.
\end{enumerate}
Observe that the above conditions imply
\begin{enumerate}
\item[(3)] $h_1, h_2, h_3$ are pairwise distinct.
\end{enumerate}
Indeed, if $h_j=h_k$ with $j\neq k$, then $h_jB_j$ and $h_jB_k$ are disjoint sets
with disjoint complements. This implies that $G=h_jB_j\sqcup h_jB_j$, and 
hence $G=B_j\sqcup B_k$, which
contradicts the fact that $B_1,B_2,B_3$ are pairwise disjoint and nonempty. 

Let $S\subseteq G$ be a set containing exactly one representative 
of each right coset in $H\backslash G$, so that
$G=\bigsqcup_{s\in S}H s$. For $j=1,2,3$, set  $A_j=\bigsqcup_{s\in S}B_j s$. We claim that
\begin{enumerate}
\item[(a)] the sets $dA_j$, for $d\in D$ and $j=1,2,3$, are pairwise disjoint, and
\item[(b)] the sets $G\setminus h_jA_j$, for $j=1,2,3$, are pairwise
disjoint.
\end{enumerate}
To see (a), let $d,e\in D$ and let $j,k\in \{1,2,3\}$.
Using that $Hs\cap Ht=\emptyset$ for $s,t\in S$ with $s\neq t$, we get 
\[dA_j\cap eA_k=\Big(\bigsqcup_{s\in S}dB_j s\Big) \cap \Big(\bigsqcup_{t\in S}eB_k t\Big)=
\bigsqcup_{s,t\in S} \underbrace{dB_j s \cap eB_kt}_{\subseteq Hs\cap Ht}= \bigsqcup_{s\in S} (dB_j \cap eB_k)s.\]
If the above intersection is nonempty, then we must 
have $dB_j\cap eB_k\neq \emptyset$, which implies
$d=e$ and $j=k$ by (1). To check condition (b), let $j,k\in \{1,2,3\}$.
Arguing as above, we have
\begin{align*}(G\setminus h_jA_j)\cap (G\setminus h_kA_k)&=
\Big(\bigsqcup_{s\in S}(Hs\setminus h_jB_j s) \Big) \cap \Big(\bigsqcup_{t\in S}(Ht\setminus h_kB_k t)\Big)\\
&=
\bigsqcup_{s\in S}\big((H\setminus h_jB_j)\cap (H\setminus h_kB_k)\big)s,\end{align*}
which is empty if $j\neq k$ by (2).


Set $\varepsilon=\frac{1}{24}$, and use amenability
of $G\acton X$ to find a continuous function 
$\mu\colon X\to \mathrm{Prob}(G)$ satisfying
\begin{equation}\label{eq:measure}\tag{3.5}
\sup_{x\in X}\|\mu(g\cdot x)-g\cdot \mu(x)\|_1<\varepsilon
\end{equation}
for all $g\in \{h^{-1}, h_1, h_2, h_3\}$.
For $j=1,2,3$, set 
\[V_j=\Big\{x\in X\colon \mu(x)(A_j)>\frac{1}{2}+\varepsilon\Big\},\]
and define $V=V_1\cup V_2\cup V_3$.
By \eqref{eq:measure}, for $d\in D=\{1,h\}$ 
we have 
\[dV_j\subseteq \Big\{x\in X\colon \mu(x)(dA_j)>\frac{1}{2}\Big\},\]
for $j=1,2,3$. Since the sets 
$dA_j$, for $d\in D$ and $j=1,2,3$,
are pairwise disjoint, it follows that 
for each $x\in X$ at most one of these sets can 
have $\mu(x)$-measure more than $\frac{1}{2}$. 
We deduce that the sets $dV_j$, 
for $d\in D$ and $j=1,2,3$, are pairwise disjoint.
In particular, we have:
\be\item[(i)] $V=V_1\sqcup V_2\sqcup V_3$.  
\item[(ii)] $V\neq X$, since $hV\cap V=\emptyset$.
\ee

We will now show that $X\prec V$.
For $j=1,2,3$, set 
\[W_j=\Big\{x\in X\colon \mu(x)(G\setminus h_jA_j)<\frac{1}{2}-2\varepsilon\Big\}.\]
We claim that 
\be\item[(iii)]$W_1\cup W_2\cup W_3=X$.
\item[(iv)] $h_j^{-1}W_j\subseteq V_j$ for $j=1,2,3$. \ee

Note that $\frac{1}{2}-2\varepsilon>\frac{1}{3}$. 
By condition (b) above, 
for every $x\in X$ at least one of either 
$G\setminus h_1A_1$, $G\setminus h_2A_2$ or $G\setminus h_3A_3$ must 
have $\mu(x)$-measure less than $\frac{1}{2}-2\ep$.
This implies (iii). To show (iv), fix $j=1,2,3$.
Given $x\in X$, the fact that $\mu(x)$ is a probability
measure on $G$ implies that 
\[\mu(x)(G\setminus A_j)=1-\mu(x)(A_j).\]
Using the above at the second step, we get
\begin{align*}
h_j^{-1}W_j&\stackrel{\eqref{eq:measure}}{\subseteq}\Big\{x\in X\colon \mu(x)(G\setminus A_j)<\frac{1}{2}-\varepsilon\Big\}\\
& =
 \Big\{x\in X\colon \mu(x)(A_j)>\frac{1}{2}+\varepsilon\Big\}=V_j,
\end{align*}
as desired.

To simplify the notation, we set $g_j=h_j^{-1}$
for $j=1,2,3$.
Let $f_1,f_2,f_3\in C(X)$
be a partition of unity subordinate to the open cover $\{W_1,W_2,W_3\}$ of $X$; see (iii) above.
Denote by $\alpha$ the action of $G$ on $C(X)$
induced by the given action $G\acton X$.
For $g\in G$, we write $u_g\in C(X)\rtimes G$ for the 
canonical unitary satisfying 
$u_gf=\alpha_g(f)u_g$ for all $f\in C(X)$.
We denote by $E\colon C(X)\rtimes G\to C(X)$ the 
canonical conditional expectation, which is determined
by $E(u_g)=0$ whenever $g\in G\setminus\{1\}$.

Set $v=\sum_{j=1}^3 \alpha_{g_j}(f_j^{1/2})u_{g_j}$, and 
note that $v=\sum_{j=1}^3 u_{g_j}f_j^{1/2}$. 
Since $\alpha_{g_j}(f_j^{1/2})$ is supported on 
$g_jW_j$ and $g_jW_j\cap g_kW_k=\emptyset$ whenever $j\neq k$ by (iv) and (i) above,
we have
\begin{equation}\label{eqn:orth}\tag{3.6}
\alpha_{g_j}(f_j^{1/2})\alpha_{g_k}(f_k^{1/2})=0
\end{equation}
whenever $j\neq k$. Using this, we get
\begin{align*}
v^*v&= \sum_{j,k=1}^3 u_{g_j}^*\alpha_{g_j}(f_j^{1/2})\alpha_{g_k}(f_k^{1/2})u_{g_k}\\
&\stackrel{\eqref{eqn:orth}}{=}
\sum_{j=1}^3 u_{g_j}^*\alpha_{g_j}(f_j)u_{g_j}
=\sum_{j=1}^3 f_j=1.
\end{align*}
Thus $v$ is an isometry. On the other hand, we have
\begin{align*}
vv^*= \sum_{j,k=1}^3 u_{g_j}f_j^{1/2}f_k^{1/2}u_{g_k}^*
=\sum_{j,k=1}^3 \alpha_{g_j}(f_j^{1/2}f_k^{1/2})u_{g_jg_k^{-1}}.
\end{align*} 
We will show that $vv^*\neq 1$ by showing that $E(vv^*)\neq 1$.
We apply $E$ to 
the expression above and use (3) to get
\[E(vv^*)=\sum_{j=1}^3 \alpha_{g_j}(f_j).\]
In particular, $E(vv^*)$ 
is supported on $\bigsqcup_{j=1}^3 g_jW_j\subseteq V$. 
Since $V\neq X$ by (ii), the above expression cannot equal 1,
and hence $v$ is not a unitary, as desired.
\end{proof}

We point out that the argument used in the theorem above is 
somewhat different from the one used to prove 
\autoref{thm:groupstospaces}. Indeed, the reasoning used in 
\autoref{thm:groupstospaces} would only give the existence of a
nontrivial open set $V$ satisfying $X\prec_1 V$, 
which suffices to show infiniteness of $M_2(C(X)\rtimes G)$,
but not of $C(X)\rtimes G$. In order to obtain $X\prec_0 V$,
we need the strengthening of 2-paradoxical towers proved for 
$\mathbb{F}_2$ in \autoref{prop:F2}.

\section{Examples of groups with paradoxical towers}\label{sec:examples}

In this section we exhibit large classes of nonamenable
groups which admit paradoxical towers; see Theorem~\ref{thm:introExamples}
in the introduction. 
In addition to proving some preservation properties 
for the class
of groups admitting paradoxical towers,
the main 
tool to construct such groups is given in 
\autoref{thm:BdryParadTowers}, where we
show that one can produce paradoxical towers in groups 
admitting \emph{some} 
topologically free $n$-filling action on a completely 
metrizable (but not necessarily locally compact) space. 
Using this, we give several concrete and explicit examples. 

We begin by looking at extensions of groups with paradoxical
towers, both by finite groups (\autoref{lem:finiteNormalSubgroup}
and by other groups with paradoxical towers (\autoref{lma:ExtParad}).

\begin{prop}\label{lem:finiteNormalSubgroup}
Let $n\in\N$, let $G$ be a group, and let $K\leq G$ a finite 
normal subgroup such that $G/K$ has $n$-paradoxical towers. 
Then $G$ has $n|K|$-paradoxical towers. 
\end{prop}

\begin{proof}
Denote by $\pi\colon G\to G/K$ the quotient map, let $F\subseteq G$ be a finite subset, and set $D_0=\pi(F)$ and $m=|K|$.
Since $G/K$ is infinite, there exist $t_1,\ldots,t_m\in G/K$ with $t_1=1$
such that $D_0t_j$, for $j=1,\ldots,m$, are pairwise disjoint sets.
Set 
\begin{equation}\tag{4.1}\label{disjtrans}
D=\bigsqcup_{j=1}^m D_0t_j,
\end{equation}
which is a finite subset of $G/K$.
Since $G/K$ admits $n$-paradoxical towers, there are sets $A_1,\dotsc,A_n\subseteq G/K$ and elements $h_1,\dotsc,h_n\in G/K$ such that 
\begin{itemize}
\item[(i)] \label{disjointinG/K} the sets $d 
A_i$, for $d\in D$ and $i=1,\dotsc,n$, are pairwise disjoint, and
\item[(ii)] \label{coverG/K}$G/K=\bigcup_{i=1}^nh_iA_i$.
\end{itemize} 
Fix an arbitrary enumeration $K=\{k_1,\ldots,k_m\}$.
Let $s\colon G/K\to G$ be a 
section for $\pi$. For $i=1,\ldots,n$ and $j=1,\ldots,m$, set 
\begin{equation}\tag{4.2}\label{df:Aijgij}
C_{i,j}=s(t_j {h_i}^{-1})s(h_iA_i) \subseteq G \ \ \mbox{ and } \ \ 
g_{i,j}=k_js(t_j {h_i}^{-1})^{-1}\in G. 
\end{equation}

We claim that the above are paradoxical towers in $G$ for $F$.
To check the first condition in \autoref{df:ParadoxicalTower}, let $d,d'\in F$, let 
$i,i'=1,\ldots,n$ and let 
$j,j'=1,\ldots,m$ satisfy 
\begin{equation}\tag{4.3}\label{nonemptyintersection}
dC_{i,j}\cap d'C_{i',j'}\not=\emptyset.
\end{equation}
Applying $\pi$ in the equation above and using \eqref{df:Aijgij}, 
we obtain 
\[\pi(d)t_jA_i\cap \pi(d')t_{j'}A_{i'}\not=\emptyset.\] 
Note that $\pi(d)t_j$ and $\pi(d')t_{j'}$ belong to $D$.
By condition (i), 
we deduce that $\pi(d)t_j=\pi(d')t_{j'}$ and $i=i'$.
The identity in \eqref{disjtrans} implies 
that $j=j'$ and $\pi(d)=\pi(d')$. In other words, $d^{-1}d'$ belongs to $K$.
Combining this with
\eqref{nonemptyintersection} and \eqref{df:Aijgij}, we get
\[d s(t_j {h_i}^{-1})s(h_iA_i)\cap d's(t_j {h_i}^{-1})s(h_iA_i)\not=\emptyset.\]
and thus 
\[s(h_iA_i) \cap \hspace{-.5cm}\underbrace{s(t_j {h_i}^{-1})^{-1}\overbrace{d^{-1}d'}^{\in K} s(t_j {h_i}^{-1})}_{ \ \ \ \ \ \ \ \ \ \ \ \ \ \ \ \ \ \ \ \ \ \in K, \ 
\mathrm{since}~K~\mathrm{is~normal}}\hspace{-.5cm}s(h_iA_i) \not=\emptyset.\]
Since $s$ is a section, we have $k s(G/K)\cap s(G/K)\not= \emptyset$ 
for some $k\in K$ if and only if $k=1$. It follows that 
$s(t_j {h_i}^{-1})^{-1}d^{-1}d' s(t_j {h_i}^{-1})=1$ and thus $d=d'$, as desired. 

We check the second condition in \autoref{df:ParadoxicalTower}.
By condition (ii), we have $s(G/K)=\bigcup_{i=1}^n s(h_iA_i)$. Using that
$G=\bigsqcup_{j=1}^m k_js(G/K)$ at the last step, we obtain
\[
\bigcup_{i,j=1}^{n,m} g_{i,j}C_{i,j}\stackrel{\eqref{df:Aijgij}}{=}\bigcup_{i,j=1}^{n,m}
k_js(t_j {h_i}^{-1})^{-1}s(t_j {h_i}^{-1}) s(h_iA_i)=
\bigcup_{i,j=1}^{n,m}k_j s(h_iA_i)=G.
\]
This proves that $G$ has $nm$-paradoxical towers, as desired. 
\end{proof}

\begin{prop}\label{lma:ExtParad}
Let $G$ be a group, let $K\leq G$ be a normal subgroup.
Assume that $G/K$ has $n$-paradoxical towers and that $K$ has $m$-paradoxical
towers. Then $G$ has $nm$-paradoxical towers.
\end{prop}
\begin{proof}
Let $\pi\colon G\to G/K$ be the canonical quotient map,
and $s\colon G/K\to G$ be any section for it.
Let $F\subseteq G$ be a finite subset, and set $E_0=F^2\cap K$ and $D=\pi(F)$.
Without loss of generality, we assume that 
$F$ is symmetric and contains the identity of $G$.
Since $G/K$ has $n$-paradoxical towers, there exist subsets
$A_1,\ldots,A_n\subseteq G/K$ and group elements $h_1,\ldots,h_n\in G/K$
such that
\be\item[(a.1)] the sets $d
A_i$, for $d\in D$ and $i=1,\dotsc,n$, are pairwise disjoint,
\item[(a.2)] $G/K=\bigcup_{i=1}^nh_iA_i$.
\ee
Set $E=\bigcup_{i=1}^n s(h_i)E_0 s(h_i)^{-1}$, which by normality 
is a (finite) subset of $K$.
Since $K$ has $m$-paradoxical towers, there exist subsets
$B_1,\ldots,B_m\subseteq K$ and group elements $k_1,\ldots,k_m\in K$
such that
\be\item[(b.1)] the sets $e
B_j$, for $e\in E$ and $j=1,\dotsc,m$, are pairwise disjoint,
\item[(b.2)] $K=\bigcup_{j=1}^m k_jB_j$.
\ee
For $i=1,\ldots,n$ and $j=1,\ldots,m$, set 
\[C_{i,j}= s(h_i)^{-1}B_js(h_iA_i) \ \ \mbox{ and } \ \ g_{i,j}=k_js(h_i).\]
Note that $\pi(C_{i,j})=A_i$ for all $i=1,\ldots,n$ and $j=1,\ldots,m$.

We claim that the above are paradoxical towers in $G$ for $F$.
It is immediate to check that $\bigcup_{i,j=1}^{n,m}g_{i,j}C_{i,j}=G$ using
(a.2) and (b.2). Given
$f,f'\in F$, $i,i'=1,\ldots,n$ and $j,j'=1,\ldots,m$, suppose that
\[fC_{i,j}\cap f'C_{i',j'}\neq \emptyset.\]
Applying $\pi$ gives $\pi(f)A_i\cap \pi(f')A_{i'}\neq \emptyset$, which
by condition (a.1) implies that $i=i'$ and $f^{-1}f'\in K$.
Set $e=f^{-1}f'\in E_0$. Substituting in the equation above, we get
\[
 s(h_i)^{-1}B_js(h_iA_i) \cap es(h_i)^{-1}B_{j'}s(h_iA_i)\neq\emptyset.
\]
Choose $a,a'\in A_i$, $b\in B_j$ and $b'\in B_{j'}$ such that 
\[s(h_i)^{-1}bs(h_ia)= es(h_i)^{-1}b's(h_ia').\]
Applying $\pi$ to the identity above gives $a=a'$, so we get 
\[s(h_i)^{-1}b= es(h_i)^{-1}b',\]
which implies that $B_j\cap (s(h_i)es(h_i)^{-1})B_{j'}\neq \emptyset$
since $F$ contains the identity of $G$.
Since $s(h_i)es(h_i)^{-1}$ belongs to $E$, condition (b.1) implies
that $s(h_i)es(h_i)^{-1}=1$ and $j=j'$. Thus $e=1$ and $f=f'$, as desired.
\end{proof}

\begin{lma}\label{rem:LocalParadoxicalTowers}
Let $n\in\N$, and let $G$ be a group that can be 
expressed as an increasing union $G=\bigcup_{k\in\N} G_k$ of 
groups $G_k$ that admit $n$-paradoxical towers. Then $G$ admits 
$n$-paradoxical towers.
\end{lma}
\begin{proof}
Let $D\subseteq G$ be a finite subset, and find $k\in\N$ such 
that $D\subseteq G_k$. Find subsets
$B_1,\dotsc,B_n\subseteq G_k$ and $g_1,\dotsc, g_n\in G_k$ such that 
\begin{enumerate}
\item the sets $dB_i$, for $d\in D$ and $i=1,\ldots,n$, are pairwise disjoint, and
\item $\bigcup_{i=1}^{n}g_i B_i=G_k$.
\end{enumerate}

Let $S\subseteq G$ be a set containing exactly one representative 
of each right coset in $G_k\backslash G$, so that
$G=\bigsqcup_{s\in S}G_k s$. For $i=1,\ldots,n$, set  $A_i=\bigsqcup_{s\in S}B_i s$. It is then 
immediate to check that $A_1,\ldots,A_n\subseteq G$ and $g_1,\ldots, g_n\in G$ satisfy
the conditions of \autoref{df:ParadoxicalTower}. 
\end{proof}

The following notion was introduced in
\cite{JolRob_simple_2000} for actions on compact spaces.

\begin{df}\label{df:nfilling}
Let $n\in\mathbb{N}$. An action $G\acton Z$ of a countable group 
on a Hausdorff space $Z$ is said to be \emph{$n$-filling} if for 
any nonempty open sets $U_1,\ldots, U_n\subseteq Z$, there exist 
$g_1,\ldots,g_n\in G$ such that $\bigcup_{j=1}^{n}g_iU_i=Z$.
\end{df}

In the definition above, we do not assume the space $Z$ to be 
compact, or even locally compact. It is not hard to see that if 
$G\acton Z$ is $n$-filling and $Z$ is locally compact, then $Z$ must
in fact be compact. There exist, however, interesting $n$-filling
actions on spaces that are not locally compact;
see \autoref{eg:acylindrical}.

\begin{rem}\label{rem:2fillStrBdryExtrProx}
Actions that are $2$-filling are also called \emph{strong boundary actions}, and their 
$C^*$-algebraic crossed products were studied in \cite{LacSpi_purely_1996}. They have
also been studied under the name \emph{extremely proximal actions} by 
Glasner in \cite{Gla_topological_1974} and are called \emph{extreme boundary actions} in \cite{BIO_simplicity_2020}.
\end{rem}

Recall that a topological space is called \emph{Baire} if the 
conclusion of the Baire category theorem holds: a countable
intersection of open dense subsets is dense. The class of Baire
spaces includes all locally compact Hausdorff spaces as well as
all completely metrizable ones.

\begin{prop}\label{thm:BdryParadTowers}
Let $n\in\N$ and let $G$ be a countable, infinite group. 
Assume that there exists a 
topologically free $n$-filling action of 
$G$ on a Hausdorff Baire space. 
Then $G$ admits $n$-paradoxical towers.
\end{prop}

\begin{proof}
Let $Z$ be a Baire space and let 
$G\acton Z$ be a topologically free $n$-filling action.
Let $D\subseteq G$ be a finite subset, and assume without loss of generality that $D$ contains the unit $1_G$ of $G$. Since $G\acton Z$ is topologically free, for 
every $g\in G\setminus\{1_G\}$, the open set
$U_g=\{z\in Z\colon  g\cdot z\neq z\}$ is dense in $Z$. 
Since $G$ is countable and $Z$ is Baire, 
the set $Y:=\bigcap_{g\in G\setminus\{1_G\}} U_g$ is dense in $Z$, and in particular nonempty. 
Fix $z_1\in Y$, so that 
$\mathrm{Stab}_G(z_1)=\{1_G\}$. Using that $G$ is infinite, for $i=2,\ldots,n$, 
we can choose $z_i\in Z$
recursively satisfying
\[z_i\in G\cdot z_1\setminus (D^{-1}D\cdot z_1\cup \ldots\cup D^{-1}D\cdot z_{i-1}). \]
Note that $dz_i=d'z_{i'}$ for $d,d'\in D$ and $i,i'=1,\ldots,n$
implies $d=d'$ and $i=i'$.
Since $Z$ is Hausdorff, there exist open neighborhoods $U_1,\ldots,U_n$ of $z_1,\ldots,z_n$, respectively, such that $dU_i$, for $d\in D$ and 
$i=1,\ldots,n$, are pairwise disjoint sets in $Z$. 
Since $G\acton Z$ is $n$-filling, there exist $g_1,\ldots,g_n\in G$ such that $\bigcup_{i=1}^{n} g_iU_i=Z$. For $i=1,\ldots,n$, set
\[A_i= \{g\in G\colon  g\cdot z_1\in U_i\}. \]
One checks that $A_1,\ldots, A_n$ and $g_1,\ldots, g_n$ satisfy the conditions of \autoref{df:ParadoxicalTower}.
\end{proof}

We turn to examples of groups which admit paradoxical towers. 
The first class we will consider is that of 
acylindrically hyperbolic groups, as introduced by Osin 
in~\cite[Definition~1.3]{Osi_acylindrically_2016}; see \autoref{eg:acylindrical}. 
This class includes all
nonamenable hyperbolic groups (in particular, all nonabelian 
free groups), all but finitely many mapping 
class groups, the outer automorphism group 
$\operatorname{Out}(\mathbb{F}_n)$ of $\mathbb{F}_n$, nonelementary 
CAT$(0)$-groups containing a rank-one element, 
and many fundamental groups of hyperbolic 
$3$-manifolds; see \cite[Appendix]{Osi_acylindrically_2016}.

\begin{prop}\label{eg:acylindrical}
Let $G$ be an acylindrically hyperbolic group. 
Then $G$ admits paradoxical towers.
\end{prop}
\begin{proof}
Let us first assume that $G$ has no nontrivial finite normal subgroups. 
By definition, $G$ admits a nonelementary acylindrical action 
by isometries on a (not necessarily proper) Gromov-hyperbolic space $Y$. 
By \cite[Theorem~1.2]{Osi_acylindrically_2016}, we can even assume that the action is cobounded (as we can take the space to be the Cayley graph associated
to a suitable generating set). 
Let $Z$ denote the (not necessarily compact) Gromov boundary $\del Y$ of $Y$.
Since $G$ has no nontrivial finite normal subgroups, \cite[Proposition~4.1]{AbbDah_property_2019} ensures that the induced action $G\acton Z$ 
is minimal and topologically free. 

We claim that $G\acton Z$ is $2$-filling.\footnote{
This argument is inspired by \cite[Example 2.1]{LacSpi_purely_1996}, 
but note that $\partial Y$
is not necessarily compact.}
Let $U,V\subseteq Z$ be nonempty open subsets. Since the action 
$G\acton Y$ is nonelementary, 
there exists a loxodromic element $h\in G$
(see \cite[Theorem~1.2]{Osi_acylindrically_2016}). 
Denote by $h^{-\infty}$ the repelling point of $h$.
By minimality of $G\acton Z$, there is $t\in G$ with $t\cdot h^{-\infty}\in U$.
By definition (see the paragraph before \cite[Theorem~1.1]{Osi_acylindrically_2016}), 
there is $y\in Y$ such that $(h^{-n}\cdot y)_{n\in\N}$ 
converges to $h^{-\infty}$, and $h^{-\infty}$ is
independent of $y$. In particular, $G$ acts on the limit points of loxodromic elements
by conjugation on the group, namely, for the loxodromic element $g=tht^{-1}$ we have 
$t\cdot h^{-\infty}=g^{-\infty}$.

Again by minimality, there is an element $s\in G$ and an open neighborhood $W$ of the attracting point $g^{+\infty}\in Z$ such that $s\cdot W\subseteq V$. Since $g$ is loxodromic, there is $n\in \N$ such that $g^n\cdot (Z\setminus U)\subseteq W$ (see \cite[Lemma 4.3]{Ham_bounded_2008} or the proof of 
\cite[Theorem~2B]{Tuk_convergence_1994}). In particular, we have $sg^n(Z\setminus U)\subseteq V$. 
Thus $U\cup g^{-n}s^{-1}\cdot V=Z$, which proves that $G\acton Z$ is $2$-filling. 

Note that by \cite[Proposition~3.4.18]{DasSimUrb_geometry_2017} $Z=\del Y$ is a completely metrizable space. By \autoref{thm:BdryParadTowers}, $G$ admits $2$-paradoxical towers.

If $G$ is an arbitrary acylindrically hyperbolic group, then $G$ contains a unique maximal finite normal subgroup $K(G)$ by \cite[Theorem~6.14]{DahGuiOsi_hyperbolically_2017}. 
We
claim that $G/K(G)$ is acylindrically hyperbolic. To see this, use 
\cite[Theorem~1.2]{Osi_acylindrically_2016} to find a proper, infinite, hyperbolically
embedded subgroup $H\hookrightarrow_h G$ (see \cite[Definition~2.9]{Osi_acylindrically_2016}).
By \cite[Theorem~6.14]{DahGuiOsi_hyperbolically_2017}, we have $K(G)\subseteq H$. 
By Lemma~8.3 in~\cite{DahGuiOsi_hyperbolically_2017}, 
the map $H/K(G)\to G/K(G)$ induced by $H\hookrightarrow_h G$ is 
a hyperbolic embedding. By \cite[Theorem~1.2]{Osi_acylindrically_2016}, this shows that $G/K(G)$ is acylindrically hyperbolic.

Since $G/K(G)$ clearly contains no nontrivial finite normal subgroups, it follows
from the claim above and the first part of the proof that 
$G/K(G)$ admits $2$-paradoxical towers.
We conclude from \autoref{lem:finiteNormalSubgroup} that 
$G$ admits $2|K(G)|$-paradoxical towers.
\end{proof}

The following is a concrete application of the example above to the 
work of Klisse \cite{Kli_topological_2020}.

\begin{eg}
Let $W$ be a nonamenable finite rank irreducible right-angled Coxeter group. 
Denote by $\del W$ its boundary in the sense of \cite[Definition~3.1]{Kli_topological_2020}.
Then the crossed product $C(\del W)\rtimes W$ is classifiable Kirchberg algebra.
\end{eg}

\begin{proof}
We claim that 
$W$ is acylindrically hyperbolic.
By \cite[Theorem~1.3]{Sis_contracting_2018},
it suffices to show that $W$ has a rank-one isometry and 
acts properly and cocompactly on a proper CAT$(0)$-space.
The fact that a nonamenable (also called nonaffine) Coxeter group  
acts properly and cocompactly on a proper
CAT$(0)$-space is a classical theorem of Moussong, and 
the fact that $W$ contains a rank-one isometry follows from 
\cite[Corollary~4.3 and Proposition~4.5]{CapFuj_rank_2010}. 
This proves that $W$
is acylindrically hyperbolic, and thus 
it admits paradoxical towers by \autoref{eg:acylindrical}.

The action $W\acton \del W$ is amenable by \cite[Theorem~0.2]{Kli_topological_2020}, minimal by \cite[Theorem~3.19]{Kli_topological_2020}, and 
and topologically free by \cite[Lemma~3.25]{Kli_topological_2020}.  
Thus $C(\del W)\rtimes W$ is purely infinite by \autoref{cor:PurInf}.
\end{proof}

Recall that an \emph{HNN-triple} $(G,H,\theta)$ consists of
a group $G$, a subgroup $H$, and an injective group homomorphism $\theta\colon H\to G$.
The \emph{HNN-extension} HNN$(G,H,\theta)$ associated to $(G,H,\theta)$
is the quotient of the free product $G\ast \Z=\langle G, x\rangle$ by the relation
$xh=\theta(h)x$ for all $h\in H$. The HNN-extension $\Gamma=\mathrm{HNN}(G,H,\theta)$ is said to be \emph{faithful} if its natural action on the associated Bass-Serre tree is faithful. Moreover, $\Gamma$ is said to be \emph{ascending} if either $G=H$ or $G=\theta(H)$. 

For the definition of a highly transitive group, we refer the reader to the introduction
of \cite{FMMS_characterization_2021}. 

\begin{prop}\label{prop:HNN}
Every faithful highly transitive non-ascending HNN-ex\-tension has 2-paradoxical towers.
\end{prop}
\begin{proof}
Let $(G,H,\theta)$ be a non-ascending HNN-triple with $\Gamma=\mathrm{HNN}(G,H,\theta)$ faithful and
highly transitive.
By \cite[Proposition~4.16]{BIO_simplicity_2020}, the natural action 
of $\Gamma$ on the Bass-Serre tree $T$ associated to $(G,H,\theta)$ is strongly
hyperbolic. Since this action is always minimal (see the comments before
\cite[Proposition~4.16]{BIO_simplicity_2020}), it follows from 
\cite[Lemma~3.5]{BIO_simplicity_2020} that the induced action
$\Gamma\curvearrowright \overline{\partial T}$ is 2-filling
(see \autoref{rem:2fillStrBdryExtrProx}).
By \cite[Theorem~B]{FMMS_characterization_2021}, the action
$\Gamma\curvearrowright \partial T$ is topologically free, and therefore
also $\Gamma\curvearrowright \overline{\partial T}$ is topologically free.
The result thus follows from \autoref{thm:BdryParadTowers}.
\end{proof}

A concrete and relevant class of groups covered by \autoref{prop:HNN} is that
of Baumslag-Solitar groups. These groups are never
acylindrically hyperbolic by \cite[Remark~8.4]{FMMS_characterization_2021},
and are thus not covered by \autoref{eg:acylindrical}.

\begin{eg}\label{eg:BS}
Let $m,n\in\Z$ with $|m|,|n|> 1$ and $|m|\neq |n|$. Then the associated Baumslag-Solitar 
group 
\[BS(m,n)=\langle a,b\colon ab^m=b^na\rangle\]
has 2-paradoxical towers.
\end{eg}
\begin{proof}
We identify $BS(m,n)$ as an HNN-extension as in \cite[Example~4.21]{BIO_simplicity_2020}:
we take $G=\Z=\langle a\rangle $, with $H=m\Z=\langle a^m\rangle$ and $\theta\colon m\Z \to \Z$
determined by $\theta(a^m)=a^n$; we denote this map by $\cdot\frac{n}{m}$. 
Since $|m|,|n|>1$, the HNN-extension is non-ascending.
Moreover, $BS(m,n)=$HNN$(\Z,m\Z,\cdot\frac{n}{m})$ is highly transitive by 
\cite[Proposition~8.8]{FMMS_characterization_2021}, and is faithful by \cite{Mol_91} (see also Remark~(iii) before Proposition~19 in \cite{HarPre_amalgamated_2011}). Thus the claim follows from \autoref{prop:HNN}. 
\end{proof}

A further class of groups we can treat with our methods is that of amalgamated free products. Given groups $A$ and $B$ containing a common
subgroup $C$, for each $k\geq 0$ the subgroup $C_k\subseteq C$ is
defined after \cite[Corollary 2]{HarPre_amalgamated_2011}.

\begin{eg} \label{eg:AmFreeProd}
Let $A$ and $B$ be groups containing a common subgroup $C$. 
Assume that the following conditions hold:
 \be\item $[A:C]>1$ and $[B:C]>2$.
 \item There is $k\geq 1$ such that $C_k=\{1\}$. \ee
Then the amalgamated free product $\Gamma=A\ast_C B$ has 2-paradoxical towers.
In particular, a free product $G\ast H$ of nontrivial groups with $|H|>2$ always has 
2-paradoxical towers.
\end{eg}
\begin{proof}
Let $T$ denote the Bass-Serre tree of $\Gamma$. By \cite[Proposition~19]{HarPre_amalgamated_2011}, the action of $\Gamma$ on $T$ is minimal and strongly hyperbolic. 
Therefore, by \cite[Lemma~3.5]{BIO_simplicity_2020}, the action of $\Gamma$ on $\overline{\partial T}$ is $2$-filling. 
By \cite[Proposition~19]{HarPre_amalgamated_2011} and \cite[Proposition~3.8]{BIO_simplicity_2020}, the action of $\Gamma$ on $\partial T$ is topologically free, which implies that the action of $\Gamma$ on $\overline{\partial T}$ is topologically free as well. The claim follows from \autoref{thm:BdryParadTowers}. 
\end{proof}

There is a generalization of \autoref{prop:HNN} and \autoref{eg:AmFreeProd}
to groups acting on trees, as follows:

\begin{rem}
Let a group $\Gamma$ act on a tree $T$. Assume that the action $\Gamma\curvearrowright T$ is minimal and strongly hyperbolic. Then the action $\Gamma\curvearrowright \overline{\partial T}$ is $2$-filling by \cite[Lemma 3.5]{BIO_simplicity_2020}.
Assume moreover that the fixator subgroup of every half-tree of $T$ is trivial, which is equivalent to topological freeness of the action of $\Gamma$ on $\overline{\partial T}$, by \cite[Proposition 3.8]{BIO_simplicity_2020} and \cite[Remark~2.1]{BIO_simplicity_2020}.
Then $\Gamma$ has 2-paradoxical towers by \autoref{thm:BdryParadTowers}.
\end{rem}

Next, we establish the existence of paradoxical towers in certain
lattices in Lie groups. Note that the following example covers
SL$_n(\Z)$ for $n\geq 3$ (while SL$_2(\Z)$ is covered by 
\autoref{eg:acylindrical}). 

\begin{eg}\label{eg:lattice}
Let $\Gamma$ be a lattice in a real connected semisimple Lie group $G$ without compact factors and with finite center. Then $\Gamma$ admits paradoxical towers. 
\end{eg}
\begin{proof}
Let us first assume that $G$ has trivial center.
Then the action of $\Gamma$ on the Furstenberg boundary $G/P$ of $G$ is topologically free and $n$-filling for some $n\in \N$ by \cite[Proposition 2.5, Remark~2.6]{JolRob_simple_2000} and the proof of \cite[Proposition 3.4]{AD_purely_1997}. Thus $\Gamma$ has paradoxical towers by \autoref{thm:BdryParadTowers}. 

We now treat the general case, so suppose that the center $K$ of $G$ is finite, 
and set $K_\Gamma=\Gamma\cap K$, which is a finite normal subgroup of $\Gamma$.
Then $G/K$ is a real connected semisimple Lie group without compact factors and 
with trivial center, and it is easy to see that $\Gamma/K_\Gamma\subseteq G/K$
is a lattice.
Then $\Gamma/K_\Gamma$ admits paradoxical towers by the paragraph above, 
and hence $\Gamma$ admits paradoxical towers by \autoref{lem:finiteNormalSubgroup}.
\end{proof}

The following, in combination with \autoref{cor:PurInf}, generalizes \cite[Proposition 4.2]{JolRob_simple_2000}. We refer the 
reader to \cite[Sections~3 and 4]{JolRob_simple_2000} and references therein 
for the definitions of buildings of type $\widetilde{A}_2$
and actions on them.

\begin{eg}\label{eg:building}
Let $G$ a group which acts simply transitively in a type rotating manner on the vertices of a building $\Delta$ of type $\widetilde{A}_2$. (Such groups are called \emph{$\widetilde{A}_2$-groups} in \cite{JolRob_simple_2000}.) Then $G$ admits paradoxical towers.
\end{eg}
\begin{proof}
The action of $G$ on the boundary of $\Delta$ is topologically free and $6$-filling by \cite[Theorem 3.8, Proposition 4.1]{JolRob_simple_2000}. Thus $G$ admits paradoxical towers
by \autoref{thm:BdryParadTowers}. 
\end{proof}

For the following example, we refer the reader to the discussion just 
before \cite[Proposition~3.5]{AD_purely_1997}.

\begin{eg}\label{eg:manifold}
Let $G$ is a countable subgroup of isometries of a visibility manifold $X$ with $\operatorname{vol}(X/G)<\infty$. Then $G$ admits paradoxical towers.
\end{eg}
\begin{proof}
It follows from \cite[Proposition 2.5]{JolRob_simple_2000} and \cite[Theorem 2.8, Theorem 2.2]{Ball_axial_1982} (see also \cite[p. 218]{AD_purely_1997}) that the action of $G$ on $\partial X$ is topologically free and $n$-filling for some $n$. Therefore $G$ admits paradoxical towers by \autoref{thm:BdryParadTowers}.
\end{proof}

To conclude, we give examples of nonamenable groups that do not have paradoxical towers. Observe, however, that many of these groups are covered by Theorem~\ref{thm:introCrossedProd}.

\begin{eg}\label{eg:F2Z}
Let $H$ be a nonamenable group and let $K$ be an infinite amenable group. Then $G=H\times K$ is a nonamenable group that does not have paradoxical towers. 
\end{eg}
\begin{proof}
Assume by contradiction that $G$ has $n$-paradoxical towers for some $n\in \N$. Let $D\subseteq K$ be a subset with $|D|\geq n+1$. We canonically identify $D$ with $\{1\}\times D\subseteq G$. Since $G$ has $n$-paradoxical towers, there are subsets $A_1,\dotsc,A_n\subseteq G$ and elements $g_1=(h_1,k_1),\dotsc,g_n=(h_n,k_n)\in G$ such that:
\begin{enumerate}
\item the sets $dA_i$, for $d\in D$ and $j=1,\ldots,n$, are pairwise disjoint;
\item $G=\bigcup_{j=1}^{n}g_jA_j$.
\end{enumerate}

Fix a finitely additive left-invariant probability measure $\mu$ on $K$, and denote by $\pi_K\colon G\to K$ the canonical projection.
Given $h\in H, d\in D$ and $j=1,\ldots,n$, we have
\[\pi_K((\{h\}\times K)\cap dA_j)=d\pi_K((\{h\}\times K)\cap A_j).\]
By condition (1) above, for different $d\in D$ and $j=1,\ldots,n$, 
the above sets are pairwise disjoint. 
Since $\mu$ is left-invariant, we get
\begin{equation*}
\mu\big(\pi_K((\{h\}\times K)\cap A_j)\big)\leq \frac{1}{n+1},
\end{equation*}
for all $h\in H$ and $j=1,\dotsc,n$. Denote by $1_G$ the unit of $G$.
Using condition (2) at the second step, we get
\begin{align*}
1=&\mu(K)\\
=&\mu\Big(\pi_K\Big(\Big(\{1_G\}\times K\Big)\cap\Big(\bigcup_{j=1}^n(h_j,k_j)A_j\Big)\Big)\Big)\\
=&\mu\Big((\pi_K\Big(\bigcup_{j=1}^n(h_j,k_j)\Big((\{h_j^{-1}\}\times K)\cap A_j\Big)\Big)\Big)\\
=&\mu\Big(\bigcup_{j=1}^nk_j\pi_K((\{h_j^{-1}\}\times K)\cap A_j)\Big)\\
\leq&\sum_{j=1}^n\mu\big(\pi_K((\{h_j^{-1}\}\times K)\cap A_j)\big)\\
\leq&\frac{n}{n+1},
\end{align*}
a contradiction.
\end{proof}



\end{document}